\documentclass[a4paper, 11pt]{article}
\usepackage[utf8]{inputenc}
\usepackage[T1]{fontenc}

\usepackage[
	sorting = nyt,
	backend = biber,
	hyperref = auto,
	backref = true,
	doi = false,
	url = false
]{biblatex}

\usepackage{amsthm}
\usepackage{amssymb}
\usepackage{mathtools}
\usepackage{stmaryrd}	\SetSymbolFont{stmry}{bold}{U}{stmry}{m}{n}	
\usepackage{bm}
\usepackage{nccrules}
\usepackage{mathrsfs}
\usepackage{euscript}
\usepackage{upgreek}
\usepackage{tikz-cd}
\usetikzlibrary{positioning, shapes.geometric, patterns, graphs, decorations.pathmorphing}

\usepackage{paralist}

\usepackage[colorlinks]{hyperref}
\hypersetup{
	linkcolor=blue,
	citecolor=red,
	urlcolor=teal,
	pdftitle = {Harish-Chandra D-modules for bi-Whittaker reduction},
	pdfauthor = {Wen-Wei Li}
}

\newcommand{\myfatslash}{\mathbin{\mkern-6mu\fatslash}}	

\newcommand{\Z}{\ensuremath{\mathbb{Z}}}
\newcommand{\Q}{\ensuremath{\mathbb{Q}}}
\newcommand{\R}{\ensuremath{\mathbb{R}}}
\newcommand{\CC}{\ensuremath{\mathbb{C}}}

\newcommand{\gr}{\operatorname{gr}}

\newcommand{\Sym}{\operatorname{Sym}}
\newcommand{\Ext}{\operatorname{Ext}}
\newcommand{\Tor}{\operatorname{Tor}}
\newcommand{\gldim}{\operatorname{gl.dim}}
\newcommand{\injdim}{\operatorname{inj.dim}}

\newcommand{\dd}{\mathop{}\!\mathrm{d}}
\newcommand{\Schw}{\ensuremath{\mathcal{S}}}	

\newcommand{\lrangle}[1]{\ensuremath{\langle #1 \rangle}}

\newcommand{\Stab}{\ensuremath{\mathrm{Stab}}}

\newcommand{\Hom}{\operatorname{Hom}}
\newcommand{\RHom}{\operatorname{RHom}}

\newcommand{\rightiso}{\ensuremath{\stackrel{\sim}{\rightarrow}}}

\newcommand{\cate}[1]{\ensuremath{\mathsf{#1}}}	
\newcommand{\dcate}[1]{\ensuremath{\text{-}\mathsf{#1}}}	
\newcommand{\cated}[1]{\ensuremath{\mathsf{#1}\text{-}}}	

\newcommand{\Ker}{\operatorname{ker}}
\newcommand{\Coker}{\operatorname{coker}}
\newcommand{\Image}{\operatorname{im}}
\newcommand{\dotimes}[1]{\ensuremath{\underset{#1}{\otimes}}}

\newcommand{\Lie}{\operatorname{Lie}}
\newcommand{\Ad}{\operatorname{Ad}}
\newcommand{\Spec}{\operatorname{Spec}}
\newcommand{\Gm}{\ensuremath{\mathbb{G}_\mathrm{m}}}
\newcommand{\Ga}{\ensuremath{\mathbb{G}_\mathrm{a}}}

\newcommand{\Supp}{\operatorname{Supp}}

\newcommand{\Wh}{\ensuremath{\mathbb{W}}}	
\newcommand{\Hk}{\ensuremath{\mathbb{H}}}	
\newcommand{\bpsi}{\ensuremath{\uppsi}}	
\newcommand{\Le}{{\ensuremath{\ell}}}		
\newcommand{\Ri}{{\ensuremath{\mathrm{r}}}}	
\newcommand{\LeRi}{{\ensuremath{\ell\mathrm{r}}}}		

\theoremstyle{plain}
\newtheorem{proposition}{Proposition}
\newtheorem{lemma}[proposition]{Lemma}
\newtheorem{theorem}[proposition]{Theorem}
\newtheorem{corollary}[proposition]{Corollary}
\theoremstyle{definition}
\newtheorem{definition}[proposition]{Definition}
\newtheorem{definition-theorem}[proposition]{Definition--Theorem}
\newtheorem{definition-proposition}[proposition]{Definition--Proposition}
\newtheorem{remark}[proposition]{Remark}

\numberwithin{equation}{section}
\numberwithin{proposition}{section}

\renewcommand{\emptyset}{\ensuremath{\varnothing}}	

\usepackage{amsmath}
\usepackage[nottoc]{tocbibind}

\usepackage{geometry}
\title{Harish-Chandra $D$-modules for bi-Whittaker reduction}
\author{Wen-Wei Li}
\date{}

\DeclareFieldFormat{postnote}{#1}	
\PassOptionsToPackage{hyphens,spaces}{url}  
\makeatletter
\renewcommand{\l@section}{\@dottedtocline{1}{1.5em}{2.0em}}
\renewcommand{\l@subsection}{\@dottedtocline{2}{4.0em}{3.0em}}
\makeatother

\begin{document}
	
\maketitle

\begin{abstract}
	Let $D(G)$ be the algebra of algebraic differential operators on a complex reductive group $G$. Denote by $\mathbb{W}$ the bi-Whittaker quantum Hamiltonian reduction of $D(G)$, also known as the quantum Toda lattice. In this article we define the admissible $\mathbb{W}$-modules and the special case of Harish-Chandra modules, the latter being bi-Whittaker variants of the invariant holonomic systems of Hotta--Kashiwara. We then study their torsion properties up to completion relative to the Kazhdan filtration, which in general is unbounded in both directions, through the geometry of universal centralizer. In the case of regular infinitesimal characters, the corresponding $D(G)$-module is shown to be the minimal extension of an irregular connection on the open Bruhat cell. Certain ring-theoretic properties of the completion of $\mathbb{W}$ are also obtained.
\end{abstract}

{\scriptsize
	\begin{tabular}{ll}
		\textbf{MSC (2020)} & Primary 13N10; Secondary 16S32, 20G05 \\
		\textbf{Keywords} & quantum Hamiltonian reduction, holonomic systems, Harish-Chandra D-modules.
	\end{tabular}
}
	
\setcounter{tocdepth}{1}
\tableofcontents

\section{Introduction}
\subsection{Background}
Let $G$ be a complex connected reductive group with Borel subgroup $B$, and let $\bpsi$ be a non-degenerate character of $\mathfrak{n} := \Lie(N)$ where $N = \mathcal{R}_{\mathrm{u}}(B)$, i.e.\ $\bpsi$ is nonzero on each simple root vector. Let us distinguish the left and right translations on $G$ using indices $\Le$ and $\Ri$, so that $G_{\Le} \times G_{\Ri}$ acts on the right of $G$. The \emph{quantum Toda lattice}, or the algebra of bi-Whittaker differential operators on $G$ is $\Wh := (D(G)/D(G)\mathfrak{n}_{\LeRi}^{\bpsi})^{N_{\Le} \times N_{\Ri}}$, where $D(G) = \Gamma(G, \mathscr{D}_G)$ is the algebra of algebraic differential operators on $G$, and $\mathfrak{n}_{\LeRi}^{\bpsi}$ is generated by bilateral Whittaker conditions $\xi - \bpsi(\xi)$ for all $\xi \in \mathfrak{n}_{\Le}$ and $\mathfrak{n}_{\Ri}$. This is a basic example of (twisted) quantum Hamiltonian reduction; an equivalence à la Skryabin identifies $\Wh\dcate{Mod}$ with the abelian category of $D(G)$-modules on which $\mathfrak{n}_{\LeRi}^{\bpsi}$ acts locally nilpotently, see \eqref{eqn:monodromic-equivalence}.

The derived category of $\Wh\dcate{Mod}$ identifies with $\mathcal{D}(N, \bpsi \backslash G / N, \bpsi)$, and plays a prominent role in representation theory, for which we refer to \cite{BG17} for a higher perspective. In contrast, the motivation of the present work comes from harmonic analysis on real groups, and we will focus on the non-derived setting.

Specifically, suppose that $G$ and $N$ are both defined over $\R$, then $\Wh$ acts on $\mathfrak{n}_{\LeRi}^\bpsi$-invariant tempered distributions on $G(\R)$. A key example is the Bessel distribution $\mathcal{B}_\pi$ attached to an admissible representation $\pi$ of $G(\R)$ as considered by Baruch \cite{Ba01}. This is one of the earliest examples of relative characters \cite[Definition 1.2]{AGM16} of admissible representations. It seems reasonable to expect that the differential equations satisfied by Bessel distributions will be crucial for studying their analytic properties, such as local integrability, see eg.\ \cite{GL04}.

Such systems of differential equations are encoded by $\Wh$-modules. Let $\mathfrak{c} = \mathfrak{g}^* \sslash G$. For all $\nu \in \mathfrak{c}$ let $\mathfrak{m}_\nu$ be the corresponding maximal ideal of $\mathcal{Z}(\mathfrak{g}) := Z(\mathcal{U}(\mathfrak{g}))$. The natural homomorphism $\mathcal{Z}(\mathfrak{g}) \to D(G)$ passes to $\Wh$. Define
\begin{align*}
	\mathcal{G}_\nu & := \Wh/\Wh\mathfrak{m}_\nu, \\
	\mathcal{G}^\natural_\nu & := D(G) \big/ \left( D(G) \mathfrak{n}_{\LeRi}^{\bpsi} + D(G)\mathfrak{m}_\nu \right),
\end{align*}
which match each other under the Skryabin equivalence. In the real setting, $\mathcal{G}^\natural_\nu$ encodes the differential equations satisfied by $\mathcal{B}_\pi$ when $\pi$ is of infinitesimal character $\nu$. More generally, one may consider \emph{admissible} $\Wh$-modules, namely the finitely generated $\Wh$-modules on which $\mathcal{Z}(\mathfrak{g})$ acts locally finitely.

Let us call $\mathcal{G}_\nu$ the \emph{Harish-Chandra $\Wh$-module} attached to $\nu \in \mathfrak{c}$. Harish-Chandra modules (or the corresponding $D(G)$-modules) can be seen as bi-Whittaker versions of the invariant holonomic system considered by Hotta--Kashiwara \cite{HK84}, in which the role of $\mathcal{B}_\pi$ is replaced by the Harish-Chandra character of $\pi$. Similar ideas were pursued for certain infinitesimal symmetric spaces in Sekiguchi \cite{Sek85}, Levasseur--Stafford \cite{LS99}, Laurent \cite{La03}, Galina--Laurent \cite{GL04} and Bellamy--Nevins--Stafford \cite{BNS24}.

In all the works cited above, the Harish-Chandra $\mathscr{D}$-modules are regular holonomic, and they turn out to be the minimal extensions of flat connections on a certain principal open subset $U$. This implies the key fact that the invariant eigendistributions (eg.\ Harish-Chandra characters) in their settings are determined by restriction to $U$, see \cite[Corollary 1.3 (2)]{BNS24}.

In the bi-Whittaker setting, $\mathcal{G}^\natural_\nu$ is holonomic, and smooth over the big Bruhat cell $Bw_0 B$ (Proposition \ref{prop:connection-big-cell}), but regularity fails due to the exponential behavior incurred by $\bpsi$. On the other hand, the determination of Bessel distributions by restriction to $Bw_0 B$ has been proven by \cite{AGK15, Lu22} in a broader and simpler framework, without resort to $D$-modules.

Nonetheless, one can still ask whether $\mathcal{G}^{\natural}_\nu$ is the minimal extension of some flat connection on $Bw_0 B$. Such information might be useful for further studies of Bessel distributions, for example the local integrability \cite{GL04}, although these discussions lie beyond the scope of the present article.

\subsection{Main results and idea of proofs}
Fix a principal $\mathfrak{sl}(2)$-triple $(\mathsf{e}, \mathsf{h}, \mathsf{f})$ such that $\mathsf{e}$ corresponds to $\bpsi$ via a chosen invariant form on $\mathfrak{g}$ and $\mathsf{f} \in \mathfrak{n}$. The complement of $Bw_0 B$ in $G$ is defined by some $d \in \CC[G]^{N \times N}$ attached to a regular dominant weight; let $\mathcal{S} = d^{\Z_{\geq 0}}$. Also note that $\CC[G]^{N \times N}$ maps into $\Wh$.

The modest goal of the present work is to understand the admissible $\Wh$-modules, and try to see if the corresponding $D(G)$-module equals a minimal extension from $Bw_0 B$. This amounts to the absence of nonzero submodules or quotients of $\mathcal{S}$-torsion, i.e.\ modules in which every element is annihilated by some element of $\mathcal{S}$. This condition can be tested on the level of $\Wh$-modules (Lemma \ref{prop:support-torsion}).

Our main tool is Ginzburg's results \cite{Gin18}. Among others, he introduced an ascending $\Z$-filtration $\mathrm{F}_\bullet \Wh$ on $\Wh$, deduced from the order filtration on $D(G)$ adjusted by $\mathrm{ad}(\mathsf{h})$-weights (Kazhdan filtration). Moreover, $\Wh$ is identified with the spherical subalgebra of the \emph{degenerate nil-DAHA} (double affine Hecke algebra) associated with $G$, which is first defined in \cite{KK86}.

In general, $\mathrm{F}_\bullet \Wh$ is exhaustive, separating yet unbounded in both directions. Let $\mathfrak{Z}$ be the \emph{universal centralizer} of $G$, and $\kappa: \mathfrak{Z} \to \mathfrak{c}$ the characteristic morphism. There is an isomorphism of Poisson algebras $\gr \Wh \simeq \CC[\mathfrak{Z}]$ inducing $\mathcal{Z}(\mathfrak{g}) \simeq \kappa^* \CC[\mathfrak{c}]$. These statements are given for simply connected $G$ in \textit{loc.\ cit.}, but they carry over to the general case by taking isogenies.

For every finitely generated $\Wh$-module $M$, let $\widehat{M}$ (resp.\ $\gr M$) be its completion (resp.\ graded module) with respect to any good filtration relative to $\mathrm{F}_\bullet \Wh$.

\begin{theorem}[= Corollary \ref{prop:torsion-cor}]
	\label{prop:torsion-cor-preview}
	If $M$ is an admissible $\Wh$-module and $R$ is a subquotient of $M$ of $\mathcal{S}$-torsion, then $\widehat{R} = 0$, or equivalently $\gr(R) = 0$.
\end{theorem}

The strategy is inspired by \cite{LS99, BNS24}, namely by showing $\mathrm{SS}(R) = \emptyset$ from the fact that the singular support $\mathrm{SS}(R)$ of $R$ is co-isotropic in the symplectic variety $\mathfrak{Z}$. There are two main differences, however.
\begin{itemize}
	\item We have to cope with unbounded $\Z$-filtrations and the associated completions for non-commutative algebras. Many standard results apply to this setting by making systematic use of Rees algebras; a summary is given in Section \ref{sec:filtrations} due to the lack of an adequate reference.
	
	Nevertheless, for unbounded filtrations it may happen that $R \neq \{0\}$ whilst $\widehat{R} = \{0\}$. This annoying phenomenon is excluded for \emph{Zariskian} $\Z$-filtered algebras (see eg.\ \cite[Section 1]{ABO89} for the general notion), such as $\widehat{\Wh}$; however $\Wh$ itself is \textsc{not} Zariskian unless $G$ is a torus as shown in Proposition \ref{prop:non-Zariskian}. This explains the appearance of completions for the statement.
	
	\item To prove that $\mathrm{SS}(R) = \emptyset$, we shall show that $\dim \mathrm{SS}(R) < \frac{1}{2} \dim \mathfrak{Z} = \dim \mathfrak{c}$. It boils down to showing that the big cell $\mathfrak{B}_{w_0}$ intersects all connected components of each fiber $\kappa^{-1}(\nu)$ (Lemma \ref{prop:fiber-cell}), where $\mathfrak{B}_{w_0}$ is the preimage of $Bw_0 B$ under a natural map $\mathfrak{Z} \to B \backslash G / B$ of sets. This key geometric fact is established by Kostant \cite{Ko79} and independently by Jin \cite[Section 5.2]{Jin22}.
\end{itemize}

On the positive side, it will be shown in Proposition \ref{prop:Gnu-nonzero} that $\widehat{\mathcal{G}}_\nu \neq \{0\}$ by leveraging finer information from \cite{Gin18}. With this input, it is routine to deduce the following ring-theoretic results.

\begin{proposition}[= Propositions \ref{prop:Auslander-regular} + \ref{prop:gldim}]
	The ring $\widehat{\Wh}$ is Auslander-regular with $\injdim \widehat{\Wh} = n = \gldim \widehat{\Wh}$ where $n$ is the reductive rank of $G$. 
\end{proposition}

\begin{proposition}[= Proposition \ref{prop:admissible-hh} + Theorem \ref{prop:Iwanaga}]
	The completions of admissible $\Wh$-modules are holonomic $\widehat{\Wh}$-modules in the ring-theoretic sense (Definition \ref{def:holonomic}). The functor $\Ext_{\widehat{\Wh}}^n(\cdot, \widehat{\Wh})$ gives an equivalence between the categories of left and right holonomic $\widehat{\Wh}$-modules.
\end{proposition}

The above are based on the Zariskian property of $\widehat{\Wh}$. As for $\Wh$-modules, one can only infer that $\Ext^n_{\Wh}(\cdot, \Wh)$ preserves admissibility up to completions and extensions. See Proposition \ref{prop:Ext-admissible} for a precise statement.

Next, enhance $B$ to a Borel pair $(B, T)$ so that $\mathfrak{c} \simeq \mathfrak{t}^* / \mathrm{W}$. Here $\mathrm{W} = \mathrm{W}(G, T)$ is the Weyl group acting linearly on $\mathfrak{t}^*$. Let $\widetilde{\mathrm{W}} := \mathrm{W} \ltimes \mathbf{X}^*(T)$, which acts on $\mathfrak{t}^*$ by affine transformations. Following a familiar paradigm, $\cate{Adm}$ decomposes into blocks $\cate{Adm}_{[\nu]}$, where $[\nu]$ ranges over $\widetilde{\mathrm{W}}$-orbits in $\mathfrak{t}^*$ (Definition \ref{def:block}). In fact, $\cate{Adm}_{[\nu]}$ is generated by $\mathcal{G}_\mu$ where $\mu$ ranges over the $\mathrm{W}$-orbits in $[\nu]$.

\begin{theorem}[= Theorem \ref{prop:simplicity}]
	Suppose that $\nu \in \mathfrak{c}$ is $\widetilde{\mathrm{W}}$-regular in the sense that some (equivalently, any) representative $\dot{\nu} \in \mathfrak{t}^*$ has trivial $\widetilde{\mathrm{W}}$-stabilizer. Then $\mathcal{G}_\nu$ is a simple $\Wh$-module.
\end{theorem}

The above can probably be deduced from the deep results in \cite[(1.2.2), (1.2.3)]{Lo18} or \cite[Theorem 1.5.1]{Gin18}. We offer a direct proof by realizing $\Wh$ inside some localization of $D(T)^{\mathrm{W}}$, based on a similar construction in \cite{Gin18} involving degenerate nil-DAHA. The following consequence is almost immediate.

\begin{theorem}[= Theorems \ref{prop:decompletion-reg} + \ref{prop:minimal-extension}]
	Suppose $\nu$ is $\widetilde{\mathrm{W}}$-regular. Let $M$ be an object of $\cate{Adm}_{[\nu]}$ equipped with a good filtration. Then
	\[ \widehat{M} = 0 \iff \gr(M) = 0 \iff M = 0. \]
	
	Moreover, let $M^\natural$ be the $D(G)$-module corresponding to $M$. Denote by $j$ the open embedding $Bw_0 B \hookrightarrow G$. Then $j^* M^\natural$ is a flat connection on $Bw_0 B$, and there is a canonical isomorphism of $D(G)$-modules
	\[ M^\natural \rightiso j_{!*} \left( j^* M^\natural \right). \]
\end{theorem}

Indeed, the first part allows one to ``decomplete'' Theorem \ref{prop:torsion-cor-preview}. A natural question is whether such a decompletion works for admissible $\Wh$-modules $M$ in general. Note that this does not contradict the fact that $\mathrm{F}_\bullet \Wh$ is non-Zariskian.

\subsection{Outline of this article}
This article is organized as follows.

In Section \ref{sec:filtrations} we briefly review the $\Z$-filtered non-commutative algebras, modules, their completions and Ore localizations for the reader's convenience, proving a few auxiliary results for later use. Section \ref{sec:bi-Whittaker} is a recap of the quantum Toda lattice $\Wh$ associated with a connected reductive group $G$, endowed with the filtration induced from the Kazhdan filtration. There are also discussions about their behavior under isogenies, the embedding of $\CC[G]^{N_{\Le} \times N_{\Ri}}$ and its induced filtrations.

Section \ref{sec:universal-centralizer} summarizes two constructions of the universal centralizer $\mathfrak{Z}$ associated with $G$, together with the canonical morphism $\kappa: \mathfrak{Z} \to \mathfrak{c}$. We apply the results of \cite{Ko79, Jin22} to study how the fibers of $\kappa$ intersect the big Bruhat cell $\mathfrak{B}_{w_0}$ in $\mathfrak{Z}$.

Section \ref{sec:W-properties} summarizes certain key results from \cite{Gin18} about the structure of $\Wh$ as a filtered algebra. In Section \ref{sec:W-modules} we proceed to introduce the category $\cate{Adm}$ of admissible $\Wh$-modules, and the Harish-Chandra $\Wh$-modules $\mathcal{G}_\nu$ for all $\nu \in \mathfrak{c}$. Furthermore, the strategy in \cite{LS99, BNS24} is recast to prove that admissible $\Wh$-modules are $\mathcal{S}$-torsion-free up to completion, and a decomposition of $\cate{Adm}$ according to infinitesimal characters is given. The relation to minimal extensions of $\mathfrak{n}_{\LeRi}^{\bpsi}$-monodromic $D(G)$-modules is also discussed there, modulo (de)completion.

In Section \ref{sec:homological} we show that $\widehat{\Wh}$ is an Auslander-regular ring, and determine its homological dimension using results from the previous section. We also review the ring-theoretic notion of holonomic $\widehat{\Wh}$-modules, interpret it in terms of singular supports, and deduce that $\Ext_{\Wh}^n(\cdot, \Wh)$ preserves admissibility up to completions, where $n$ denotes the reductive rank of $G$.

Finally, in Section \ref{sec:regular-inf-char} we follow the ideas from \cite{Gin18} to realize $\Wh$ inside a certain Ore localization of $D(T)^{\mathrm{W}}$, where $T \subset G$ is a maximal torus with Weyl group $\mathrm{W}$. This allows us to show that, for $\widetilde{\mathrm{W}}$-regular infinitesimal characters $[\nu]$, the Harish-Chandra $\Wh$-modules $\mathcal{G}_\nu$ are simple. Consequently, the prior results from Section \ref{sec:W-modules} can be decompleted for these blocks $\cate{Adm}_{[\nu]}$ of $\cate{Adm}$.

\subsection{Conventions}
All filtrations are indexed by $\Z$, ascending and exhaustive; the associated graded objects are denoted as $\gr(\cdots)$. A ring is said to be Noetherian if it is both left and right Noetherian. Linear duals are denoted by $V \mapsto V^*$ and the non-decorated tensor products $\otimes$ are taken over $\CC$.

Given a ring $A$, denote by $A\dcate{Mod}$ (resp.\ $\cated{Mod}A$) the category of left (resp.\ right) $A$-modules, and by $A\dcate{Mod}_{\mathrm{fg}}$ (resp.\ $\cated{Mod}A_{\mathrm{fg}}$) the subcategory of finitely generated modules. Unless otherwise specified, all modules are left modules in the sequel. The opposite category of $\mathcal{C}$ is denoted by $\mathcal{C}^{\mathrm{opp}}$.

Suppose that a group $\Gamma$ acts on an algebra $A$ via automorphisms. We denote by $\Gamma \ltimes A$ the smash product of the group bialgebra of $\Gamma$ and $A$.

For a commutative $\CC$-algebra $A$ acting on a $\CC$-vector space $U$ and a character $\nu: A \to \CC$, we say $u \in U$ is a generalized $A$-eigenvector with eigenvalue $\nu$ if some power of $\Ker(\nu)$ annihilates $u$. The symmetric algebra of $U$ is denoted as $\Sym(U)$.

The schemes under consideration are reduced $\CC$-schemes of finite type unless otherwise specified, and such a scheme $X$ will often be identified with the space $X(\CC)$ of its $\CC$-points. Denote the $\CC$-algebra of regular functions on $X$ as $\CC[X] = \Gamma(X, \mathscr{O}_X)$, and write $\CC(X)$ for the function field when $X$ is integral. For smooth $X$, denote the $\CC$-algebra of algebraic differential operators on $X$ as $D(X) = \Gamma(X, \mathscr{D}_X)$.

The linear algebraic groups are assumed to be smooth $\CC$-groups. Denote the center of such a group $G$ as $Z_G$. The enveloping algebra of $\mathfrak{g} := \Lie(G)$ is denoted by $\mathcal{U}(\mathfrak{g})$. The adjoint (resp.\ co-adjoint) action of $G$ on $\mathfrak{g}$ (resp.\ $\mathfrak{g}^*$) is denoted by $\Ad$ (resp.\ $\Ad^*$). When $G$ is connected reductive, denote by $G_{\mathrm{SC}}$ the simply connected cover of the derived subgroup $G_{\mathrm{der}} \subset G$.

Denote the additive (resp.\ multiplicative) group scheme as $\Ga$ (resp.\ $\Gm$). For a torus $T$, set $\mathbf{X}^*(T) = \Hom(T, \Gm)$ and $\mathbf{X}_*(T) = \Hom(\Gm, T)$, which embeds into $\mathfrak{t}^*$ and $\mathfrak{t}$ respectively. When $T$ is a maximal torus of a connected reductive group $G$ with Weyl group $\mathrm{W} = \mathrm{W}(G, T)$, we will consider the linear $\mathrm{W}$-action on $\mathfrak{t}$ and $\mathfrak{t}^*$ instead of the dot-action.

For a connected reductive group $G$, its reductive rank is denoted by $\mathrm{rk}(G)$. The opposite of a Borel pair $(B, T)$ of $G$ is denoted by $(B^-, T)$. If $N = \mathcal{R}_{\mathrm{u}}(B)$ is the unipotent radical of $B$, then we write $N^- = \mathcal{R}_{\mathrm{u}}(B^-)$.

Unless otherwise specified, the group actions on schemes are right actions. When $G$ acts on a smooth scheme $X$, the corresponding moment map is denoted by $\bm{\mu}: \mathrm{T}^* X \to \mathfrak{g}^*$. Categorical quotients are denoted by $X \sslash G$ whenever they exist.

\subsection{Acknowledgements}
The author is deeply grateful to Victor Ginzburg, Xin Jin and Jun Yu for answering his naive questions. This research is supported by NSFC, Grant No.\ 11922101 and 12321001.

\section{Filtered rings and modules}\label{sec:filtrations}
Fix a commutative ring $\Bbbk$. The filtrations will be written as $\mathrm{F}_\bullet A$ (for $\Bbbk$-algebras) or $\mathrm{F}_\bullet M$ (for modules over filtered algebras). For brevity, we consider only left modules; analogous results hold for right modules.

\subsection{Filtrations and completions}\label{subsec:fil-completions}
Let $A$ be a filtered $\Bbbk$-algebra; we assume $\Bbbk$ lands in $\mathrm{F}_0 A$. A \emph{filtered $A$-module} is an $A$-module $M$ equipped with filtration $\mathrm{F}_\bullet M$, such that $\mathrm{F}_n M$ is an additive subgroup of $M$ and $\mathrm{F}_m A \cdot \mathrm{F}_n M \subset \mathrm{F}_{m+n} M$, for all $m, n \in \Z$.

A homomorphism $\varphi: A \to A'$ between filtered rings is said to be filtered if $\varphi(\mathrm{F}_n A) \subset \mathrm{F}_n A'$ for all $n \in \Z$. Similarly for homomorphisms between filtered $A$-modules.

We say $A$ (resp.\ $M$) is \emph{positively filtered} if $\mathrm{F}_{-1} A = 0$ (resp.\ $\mathrm{F}_{-1} M = 0$). We say the filtration on $A$ (resp.\ $M$) is \emph{separating} if $\bigcap_n \mathrm{F}_n A = 0$ (resp.\ $\bigcap_n \mathrm{F}_n M = 0$).

Let $A$ be a filtered $\Bbbk$-algebra, and endow it with the topology induced by the filtration. The completion of $A$ can be identified with
\begin{equation*}
	\widehat{A} = \varprojlim_n A/\mathrm{F}_n A.
\end{equation*}
The $\Bbbk$-algebra $\widehat{A}$ is still filtered by setting $\mathrm{F}_m \widehat{A} = \varprojlim_{n \leq m} \mathrm{F}_m A / \mathrm{F}_n A$; this filtration is separating, and gives rise to the topology on $\widehat{A}$. Similarly for the completion $\widehat{M}$ for a filtered $A$-module $M$, which is then a filtered $\widehat{A}$-module. There are natural maps $A \to \widehat{A}$ and $M \to \widehat{M}$.

To recall the Rees construction, we introduce a formal variable $\hbar$. Given a filtered $\Bbbk$-algebra $A$, define the associated \emph{Rees algebra} as the $\Z$-graded $\Bbbk$-algebra
\begin{equation*}
	A_{\hbar} := \bigoplus_{n \in \Z} (\mathrm{F}_n A) \hbar^n \subset A[\hbar].
\end{equation*}
It becomes a graded $\Bbbk[\hbar]$-algebra by placing $\hbar$ in degree $1$. The associated $\Z$-graded $\Bbbk$-algebra of $A$ is
\begin{equation*}
	\gr A := \bigoplus_{n \in \Z} \mathrm{F}_n A / \mathrm{F}_{n-1} A = A_{\hbar}|_{\hbar = 0}.
\end{equation*}
Here we put $A_{\hbar}|_{\hbar = c} := A_{\hbar}/(\hbar - c)A$, for all $c \in \Bbbk$.

Similarly, for a filtered $A$-module $M$, the Rees construction yields the graded $A_{\hbar}$-module $M_{\hbar}$ and the associated $\Z$-graded $\gr A$-module $\gr M = M_{\hbar}|_{\hbar = 0}$. Here we put $M_{\hbar}|_{\hbar = c} := M_{\hbar} / (\hbar - c)M$, for all $c \in \Bbbk$.

\begin{lemma}\label{prop:Ah-Noetherian}
	Assume $A_{\hbar}$ is a Noetherian ring. Then:
	\begin{enumerate}[(i)]
		\item $\widehat{A}_{\hbar}$ is Noetherian,
		\item $\widehat{A}$ and $\gr A$ are both Noetherian,
		\item the homomorphism $A \to \widehat{A}$ induces $\gr A \rightiso \gr \widehat{A}$.
	\end{enumerate}
\end{lemma}
\begin{proof}
	Begin with (i). By \cite[2.8 Proposition]{ABO89}, we have
	\[ \widehat{A}_{\hbar} = (A_{\hbar})^{\wedge, \gr} := \textstyle\varprojlim_n^{\gr} \left( A_{\hbar}/\hbar^n A_{\hbar} \right), \]
	where for an inverse system of $\Z$-graded objects, $\varprojlim^{\gr}$ denotes the direct sum of the $\varprojlim$ of each graded piece. We refer to \cite[Section 2.2]{NO04} for a systematic treatment of this construction.
	
	Next, by \cite[3.4 Theorem]{McC79}, the ungraded $\hbar$-adic completion $(A_{\hbar})^\wedge$ is Noetherian. Indeed, $\hbar$ is central in $A_{\hbar}$, thus the ideal $(\hbar)$ it generates is ``polycentral'' in the sense of \cite[2.2 Definition]{McC79}. Filter each $\mathrm{F}_n A$ by $\cdots \subset \mathrm{F}_{n-1} A \subset \mathrm{F}_n A$. On the level of $\Bbbk$-modules,
	\[ (A_{\hbar})^{\wedge, \gr} = \bigoplus_{n \in \Z} (\mathrm{F}_n A)^\wedge \subset (A_{\hbar})^{\wedge} \subset \prod_{n \in \Z} (\mathrm{F}_n A)^{\wedge}. \]
	To show the ring $(A_{\hbar})^{\wedge, \gr}$ is left (resp.\ right) Noetherian, consider a chain of left (resp.\ right) ideals
	\[ I_1 \subset I_2 \subset \cdots \]
	in $(A_{\hbar})^{\wedge, \gr}$. Let $\tilde{I}_i$ be the left (resp.\ right) ideal of $(A_{\hbar})^{\wedge}$ generated by $I_i$, for each $i \geq 1$. Then there exists $p \geq 1$ such that $q > p \implies \tilde{I}_q = \tilde{I}_p$. We claim that $q > p \implies I_p = I_q$ as well.
	
	Indeed, it suffices to consider left ideals. Let $q > p$ and $x \in I_q$. Since $\tilde{I}_q = \tilde{I}_p$, in $(A_{\hbar})^{\wedge}$ one has an expression 
	\[ x = t_1 y_1 + \cdots + t_k y_k, \quad y_i \in I_p, \quad t_i \in (A_{\hbar})^{\wedge}. \]
	As $x, y_1, \ldots, y_k$ have only finitely many nonzero components when viewed inside $\prod_{n \in \Z} (\mathrm{F}_n A)^{\wedge}$, by dropping redundant components we can adjust $t_1, \ldots, t_k$ to ensure $t_i \in (A_{\hbar})^{\wedge, \gr}$ for each $i$. Hence $x \in I_p$. The claim follows.
	
	Therefore, the ring $\widehat{A}_{\hbar} \simeq (A_{\hbar})^{\wedge, \gr}$ is left (resp.\ right) Noetherian.
	
	For (ii), note that $\widehat{A} \simeq \widehat{A}_{\hbar}|_{\hbar = 1}$ is Noetherian by (i). Similarly, $\gr A \simeq A_{\hbar}|_{\hbar = 0}$ is Noetherian.
	
	Consider (iii). For each $m \in \Z$ we have
	\begin{align*}
		\mathrm{F}_m \widehat{A} & = \varprojlim_{n \leq m} \mathrm{F}_m A / \mathrm{F}_n A \\
		& = \left\{ (\bar{x}_n)_n \in \varprojlim_{n \in \Z} A/\mathrm{F}_n A \middle| n \geq m \implies \bar{x}_n = 0 \right\},
	\end{align*}
	and $\mathrm{F}_m A \to \mathrm{F}_m \widehat{A}$ is the diagonal map. Clearly $\mathrm{F}_m A /\mathrm{F}_{m-1} A \hookrightarrow \mathrm{F}_m \widehat{A} / \mathrm{F}_{m-1} \widehat{A}$. As for surjectivity, given $x = (\bar{x}_n)_n \in \mathrm{F}_m \widehat{A}$, take any preimage $y \in A$ of $\bar{x}_{m-1}$, then $y \in \mathrm{F}_m A$ and its diagonal image belongs to $x + \mathrm{F}_{m-1} \widehat{A}$, as desired.
\end{proof}

\begin{lemma}\label{prop:Mh-Noetherian}
	Assume that the filtration on $A$ is separating, and $A_{\hbar}$ is a Noetherian ring. Let $M$ be a filtered $A$-module such that $M_{\hbar}$ is finitely generated over $A_{\hbar}$, so that $M$ is finitely generated over $A$. Then:
	\begin{enumerate}[(i)]
		\item $\widehat{M}_{\hbar}$ is finitely generated over $\widehat{A}_{\hbar}$,
		\item $\widehat{M}$ is finitely generated over $\widehat{A}$, and $\gr M$ is finitely generated over $\gr A$,
		\item the homomorphism $M \to \widehat{M}$ induces $\gr M \rightiso \gr \widehat{M}$, compatibly with $\gr A \rightiso \gr \widehat{A}$.
	\end{enumerate}
\end{lemma}
\begin{proof}
	By \cite[3.11 Remarks]{ABO89}, the construction of $\widehat{M}$ (resp.\ $\widehat{M}_{\hbar}$) corresponds to the special case $S = \{1\}$ of the algebraic microlocalization $Q^\mu_S(M)$ (resp.\ its Rees counterpart) introduced in \cite[Section 3]{ABO89}. Hence \cite[3.19 Theorem]{ABO89} implies
	\[ \widehat{M}_{\hbar} \simeq \widehat{A}_{\hbar} \dotimes{A_{\hbar}} M_{\hbar} \]
	as $\widehat{A}_{\hbar}$-modules, and (i) follows.
	
	Taking quotients by the central elements $\hbar - 1$ and $\hbar$, respectively, (ii) follows. Finally, the argument for (iii) is analogous to that for Lemma \ref{prop:Ah-Noetherian} (iii); alternatively, one can also apply \cite[3.18 Corollaries (1)]{ABO89} with $S = \{1\}$.
\end{proof}

\subsection{Good filtrations}
Let $A$ be a filtered $\Bbbk$-algebra.

\begin{definition}
	Let $M$ be a filtered $A$-module. If there exists a finite family $x_1, \ldots, x_k \in M$ and $n_1, \ldots, n_k \in \Z$ such that $x_i \in \mathrm{F}_{n_i} M$ for all $1 \leq i \leq k$ and
	\[ \mathrm{F}_p M = \sum_{i=1}^k \mathrm{F}_{p - n_i} A \cdot x_i \]
	for all $p \in \Z$, then the filtration $\mathrm{F}_\bullet M$ is said to be \emph{good}. Two filtrations $\mathrm{F}_\bullet M$ and $\mathrm{F}'_\bullet M$ on $M$ are said to be equivalent if there exists $w \in \Z_{\geq 0}$ such that
	\[ \mathrm{F}_{n-w} M \subset \mathrm{F}'_n M \subset \mathrm{F}_{n+w} M \]
	for all $n \in \Z$.
\end{definition}

Equivalent filtrations induce the same topology on $M$, hence the associated completions are canonically isomorphic as $\widehat{A}$-modules.

\begin{lemma}[\protect{\cite[2.2 Lemma]{ABO89}}]
	\label{prop:good-Rees}
	The filtration $\mathrm{F}_\bullet M$ on $M$ is good if and only if the corresponding $A_{\hbar}$-module $M_{\hbar}$ is finitely generated.
\end{lemma}

\begin{lemma}
	In the setting of Lemma \ref{prop:Mh-Noetherian}, if the filtration on $M$ is good, then so is the induced filtration on $\widehat{M}$.
\end{lemma}
\begin{proof}
	Combine Lemma \ref{prop:Mh-Noetherian} (i) together with Lemma \ref{prop:good-Rees} applied to the $\widehat{A}$-module $\widehat{M}$.
\end{proof}

Below are some standard facts about filtered $A$-modules $M$ that we will need.
\begin{itemize}
	\item Any two good filtrations on $M$ are equivalent; see \cite[1.6 Proposition]{ABO89}.
	\item Let $M$ be an $A$-module with generators $x_1, \ldots, x_n$, then we can define the filtration
	\begin{equation}\label{eqn:good-filtration-fg}
		\mathrm{F}_p M := \sum_{i=1}^n (\mathrm{F}_p A)x_i
	\end{equation}
	on $M$. This is a good filtration; see \cite[Section 1.4]{ABO89}.
	\item Good filtrations induce good filtrations on quotient modules. When $A_{\hbar}$ is Noetherian, good filtrations also induce good filtrations on submodules; see \cite[2.6 Proposition]{ABO89}.
\end{itemize}

Assume $A_{\hbar}$ is Noetherian, then so are $A \simeq A_{\hbar}|_{\hbar = 1}$, $\widehat{A} \simeq \widehat{A}_{\hbar}|_{\hbar = 1}$ and $\gr(A) \simeq \gr(\widehat{A})$ by Lemma \ref{prop:Ah-Noetherian} (iii). The category $A\dcate{Mod}_{\mathrm{fg}}$ is thus Abelian, and every object therein carries a good filtration, unique up to equivalence; moreover, these filtrations are compatible with short exact sequences. The same holds for $\widehat{A}\dcate{Mod}_{\mathrm{fg}}$.

In particular, the completion of finitely generated $A$-modules makes sense. The following result is well-known in the commutative case, and the standard arguments are reproduced below.

\begin{proposition}\label{prop:completion-tensor}
	Let $A$ be a filtered $\Bbbk$-algebra such that $A_{\hbar}$ is Noetherian.
	\begin{enumerate}[(i)]
		\item Taking completion yields an exact functor $A\dcate{Mod}_{\mathrm{fg}} \to \widehat{A}\dcate{Mod}$.
		\item For every finitely generated $A$-module $M$, the canonical homomorphism $\widehat{A} \dotimes{A} M \to \widehat{M}$ is an isomorphism.
		\item the canonical homomorphism $A \to \widehat{A}$ is flat (on either side).
	\end{enumerate}
\end{proposition}
\begin{proof}
	In what follows, $\widehat{A}$ is viewed as a right $A$-module.
	
	For (i), consider a short exact sequence $0 \to M' \to M \to M'' \to 0$ in $A\dcate{Mod}_{\mathrm{fg}}$. Pick any good filtration on $M$ and endow $M'$, $M''$ with the induced filtrations, so that
	\[ 0 \to M'/\mathrm{F}_n M' \to M/\mathrm{F}_n M \to M'' / \mathrm{F}_n M'' \to 0 \]
	is exact for all $n$. The Mittag-Leffler condition holds trivially for $(M'/\mathrm{F}_n M')_n$, hence $0 \to \widehat{M}' \to \widehat{M} \to \widehat{M}'' \to 0$ is exact.
	
	For (ii), the natural map $M \to \widehat{M}$ is $A$-linear, thus induces $\widehat{A} \dotimes{A} M \to \widehat{M}$. The remaining arguments are routine: take a presentation $A^{\oplus p} \to A^{\oplus q} \to M \to 0$ where $p, q \in \Z_{\geq 0}$, and argue by using the commutative diagram with exact rows
	\[\begin{tikzcd}
		\widehat{A} \dotimes{A} (A^{\oplus p}) \arrow[r] \arrow[d] & \widehat{A} \dotimes{A} (A^{\oplus q}) \arrow[d] \arrow[r] & \widehat{A} \dotimes{A} M \arrow[r] \arrow[d] & 0 \\
		(A^{\oplus p})^\wedge \arrow[r] & (A^{\oplus q})^\wedge \arrow[r] & \widehat{M} \arrow[r] & 0.
	\end{tikzcd}\]
	
	As to (iii), by (i) we have $\Tor_1^A(\widehat{A}, M) = 0$ for all finitely generated $M$, thus for all $M$ via filtered colimits. Hence $\widehat{A}$ is flat as a right $A$-module. The left case follows by considering opposite rings.
\end{proof}

\subsection{Singular supports}\label{subsec:SS}
Let $A$ be a filtered $\Bbbk$-algebra such that $A_{\hbar}$ is Noetherian.

\begin{lemma}\label{prop:null-completion}
	Let $M$ be a finitely generated $A$-module, and let $\mathrm{F}_\bullet M$ be any good filtration, then
	\[ \widehat{M} = 0 \iff \forall n,\; \mathrm{F}_n M = M \iff \gr M = 0. \]
\end{lemma}
\begin{proof}
	The first equivalence follows from the fact that $\varprojlim_n M/\mathrm{F}_n M = 0$ implies $M = \mathrm{F}_n M$ for all $n \in \Z$, since the transition maps are surjective. For the second equivalence, recall that our filtrations are assumed to be exhaustive.
\end{proof}

\begin{remark}\label{rem:Zariskian}
	The filtration $\mathrm{F}_\bullet A$ is said to be \emph{faithful} if $\gr(M) = 0 \iff M = 0$ for all finitely generated $A$-modules $M$ equipped with good filtration. If $\mathrm{F}_\bullet A$ is complete and $\gr(A)$ is Noetherian, then \cite[1.9 Example]{ABO89} says that $\mathrm{F}_\bullet A$ is faithful; in fact, by \textit{loc.\ cit.}, $\mathrm{F}_\bullet A$ is then \emph{Zariskian}. The filtered rings encountered in this work do not meet these requirements, see Proposition \ref{prop:non-Zariskian}.
\end{remark}

Assume henceforth that $\gr(A)$ is commutative. For a finitely generated $A$-module $M$, endow $M$ with any good filtration and set
\begin{equation*}
	\begin{aligned}
		J(M) & := \sqrt{\mathrm{ann}(\gr M)}, \\
		\mathrm{SS}(M) & := \Supp(\gr M) = V(J(M)).
	\end{aligned}
\end{equation*}
We call $\mathrm{SS}(M) \subset \Spec(\gr A)$ the \emph{singular support} of $M$. Lemma \ref{prop:good-Rees} ensures that $\gr M$ is finitely generated. In fact, $J(M)$ is independent of the choice of good filtration, thus so is $\mathrm{SS}(M)$. For a short exact sequence $0 \to M' \to M \to M'' \to 0$ in $A\dcate{Mod}_{\mathrm{fg}}$, the associated graded modules still fit into a short exact sequence, thus
\[ \mathrm{SS}(M) = \mathrm{SS}(M') \cup \mathrm{SS}(M''). \]
See for example \cite[Section 1]{Ga81} in the generality we need. There is a finer notion of characteristic cycles for $M$, but it is not needed in this work.

The commutator map in $A$ descends to $\gr_p(A) \times \gr_q(A) \to \gr_{p+q-1}(A)$ for all $p, q \in \Z$, hence extends to a $\Bbbk$-bilinear pairing $\{\cdot, \cdot\}: \gr(A) \dotimes{\Bbbk} \gr(A) \to \gr(A)$ making $\gr(A)$ into a Poisson $\Bbbk$-algebra. The following result is well-known.

\begin{theorem}[\protect{O.\ Gabber \cite[Theorem 1]{Ga81}}]
	\label{prop:integrability}
	Assume $A_{\hbar}$ is Noetherian, $\gr(A)$ is commutative, and $\Bbbk \supset \Q$. For every finitely generated $A$-module $M$, the ideal $J(M)$ is involutive in the sense that $\{J(M), J(M)\} \subset J(M)$.
	
	In particular, if $\Bbbk$ is a field of characteristic zero, $\Spec(\gr(A))$ is a smooth symplectic $\Bbbk$-variety and $\{\cdot, \cdot\}$ equals the corresponding Poisson bracket, then $V(J(M))$ is co-isotropic.
\end{theorem}

We emphasize that Theorem \ref{prop:integrability} is proved in \textit{loc.\ cit.} for $\Z$-filtered rings and modules, not just for positively filtered ones. Also, it might happen that $M \neq 0$ whilst $\gr(M) = 0$, in which case $J(M) = \gr(A)$ is trivially involutive.

Finally, $\widehat{A}$ satisfies the same conditions as $A$ and $\gr(A) \simeq \gr(\widehat{A})$ by Lemma \ref{prop:Ah-Noetherian}.

\begin{proposition}\label{prop:SS-completion}
	Assume that the filtration on $A$ is separating. For all finitely generated $A$-module $M$, its completion $\widehat{M}$ is finitely generated over $\widehat{A}$ and $\mathrm{SS}(M) = \mathrm{SS}(\widehat{M})$.
\end{proposition}
\begin{proof}
	Immediate from Lemma \ref{prop:Mh-Noetherian} (iii).
\end{proof}

\subsection{Kazhdan filtrations}\label{subsec:Kazhdan-filtration}
We collect some definitions from \cite[Section 3.2]{Gin09} and \cite[Section 3.2]{Gin18} below. Let $E$ be a $\Bbbk$-module equipped with
\begin{itemize}
	\item a $\Z$-grading $E = \bigoplus_{d \in \Z} E(d)$,
	\item a filtration $\cdots \subset E_{\leq j} \subset E_{\leq j+1} \subset \cdots$, where $j \in \Z$, such that each $E_{\leq j}$ is a graded submodule of $E$.
\end{itemize}
In this situation, we say the filtration $E_{\leq \bullet}$ is compatible with the grading on $E$.

\begin{definition}
	\label{def:Kazhdan-filtration}
	Given the data above, we write
	\[ E_{\leq r} = E_{\leq \lfloor r \rfloor} \]
	for all $r \in \R$. The associated \emph{Kazhdan filtration} on $E$ is the $\Z$-filtration given by $\mathrm{F}_n E = \bigoplus_{d \in \Z} (\mathrm{F}_n E)(d)$, where
	\[ (\mathrm{F}_n E)(d) := E_{\leq \frac{n-d}{2}}(d). \]
\end{definition}

This furnishes an endofunctor of the category of filtered $\Bbbk$-modules with compatible $\Z$-gradings. The Kazhdan filtration $\mathrm{F}_\bullet E$ is not necessarily positive for positive $E_{\leq \bullet}$.

Denote by $E^{\leq}_{\hbar}$ (resp.\ $E^{\mathrm{F}}_{\hbar}$) the Rees module associated with $E_{\leq \bullet}$ (resp.\ $\mathrm{F}_\bullet E$). Similarly for the graded module $\gr^{\leq} E$ (resp.\ $\gr^{\mathrm{F}} E$).

If $A$ is a filtered $\Bbbk$-algebra with compatible grading, then $\mathrm{F}_\bullet A$ also carries a compatible grading. The same is true for filtered $A$-modules with compatible gradings.

\begin{proposition}\label{prop:regrading}
	Let $A$ be a filtered $\Bbbk$-algebra with compatible grading. There are canonical isomorphisms
	\[ A^{\leq}_{\hbar} \simeq A^{\mathrm{F}}_{\hbar}, \quad \gr^{\leq} A \simeq \gr^{\mathrm{F}} A \]
	of non-graded $\Bbbk$-algebras and $\Bbbk[\hbar]$-algebras, respectively. Similarly, for a filtered $A$-module $M$ with compatible grading, there are canonical isomorphisms
	\[ M^{\leq}_{\hbar} \simeq M^{\mathrm{F}}_{\hbar}, \quad \gr^{\leq} M \simeq \gr^{\mathrm{F}} M \]
	of non-graded modules, compatibly with the earlier isomorphisms.
	
	Furthermore, if $\gr^{\leq} A$ is commutative, then the isomorphism $\gr^{\leq} A \simeq \gr^{\mathrm{F}} A$ preserves Poisson structures.
\end{proposition}
\begin{proof}
	The isomorphisms follow by a straightforward ``regrading'' via Definition \ref{def:Kazhdan-filtration}. See the discussions before \cite[Corollary 3.2.4]{Gin09} or \cite[Section 3.3]{Gin18}.
	
	Next, assume $\gr^{\leq} A$ is commutative, then so is $\gr^{\mathrm{F}} A$. Since the Poisson bracket arises from commutators in Rees algebra, we infer after a regrading that $\gr^{\leq} A \simeq \gr^{\mathrm{F}} A$ preserves (non-graded) Poisson structures.
\end{proof}

\subsection{Ore sets}\label{subsec:Ore-sets}
We will need the formalism of Ore sets and localizations; see \cite[Chapter 2, Section 1.6]{McR01}. Let $A$ be a ring. A submonoid of $(A, \cdot)$ is also said to be a \emph{multiplicative subset} of $A$.

\begin{definition}
	We say that a multiplicative subset $\mathcal{S}$ of $A$ satisfies the \emph{right (resp.\ left) Ore condition} if for all $a \in A$ and $s \in \mathcal{S}$, there exist $a' \in A$ and $s' \in \mathcal{S}$ such that
	\[ as' = sa' \quad \text{(resp.\ $s'a = a's$)}. \]
	When both the left and right Ore conditions are met, we say $\mathcal{S}$ is an \emph{Ore set} in $A$.
\end{definition}

\begin{remark}\label{rem:Ore-zero-divisor}
	For all Ore sets encountered in this article, their elements are never zero-divisors on either side (i.e.\ $as = 0 \iff a = 0 \iff sa = 0$ for all $a \in A$ and $s \in \mathcal{S}$). We make this assumption in the sequel.
\end{remark}

Given a multiplicative subset $\mathcal{S}$ of $A$ satisfying the right (resp.\ left) Ore condition, there exists a ring $\mathcal{S}^{-1} A$ together with a homomorphism $A \to \mathcal{S}^{-1} A$ such that
\begin{itemize}
	\item $\mathcal{S}$ is mapped to $(\mathcal{S}^{-1} A)^{\times}$,
	\item every ring homomorphism $f: A \to B$ such that $f(\mathcal{S}) \subset B^{\times}$ factors uniquely through $A \to \mathcal{S}^{-1} A$.  
\end{itemize}
We call $A \to \mathcal{S}^{-1} A$ the \emph{Ore localization} of $A$ with respect to $\mathcal{S}$. It is characterized by the universal property above, and can be constructed as a set of formal right (resp.\ left) fractions $as^{-1}$ (resp.\ $s^{-1} a$) where $a \in A$ and $s \in S$, modulo the familiar equivalence relations.

In particular, if $\mathcal{S}$ is an Ore set, the localizations constructed as right and left fractions are canonically isomorphic. In view of Remark \ref{rem:Ore-zero-divisor}, the homomorphism $A \to \mathcal{S}^{-1} A$ is injective.

\begin{proposition}\label{prop:Ore-Rees}
	Consider a filtered $\Bbbk$-algebra $A$. Let $\mathcal{R}$ be multiplicative subset of $A_{\hbar}$ consisting of nonzero homogeneous elements. Fix $c \in \Bbbk^{\times}$ and denote by $\mathcal{R}_c \subset A_{\hbar}|_{\hbar=c}$ the image of $\mathcal{R}$. Suppose that $\mathcal{R}_c$ contains no zero-divisors on either side, then the same holds for $\mathcal{R}$.
	
	In addition, suppose that $\mathcal{R}$ satisfies the right (resp.\ left) Ore condition. Then:
	\begin{enumerate}[(i)]
		\item $\mathcal{R}_c$ also satisfies the right (resp.\ left) Ore property. In particular, $\mathcal{R}_c$ is an Ore set if $\mathcal{R}$ is;
		\item $\mathcal{R}^{-1} A_{\hbar}$ is still a graded $\Bbbk[\hbar]$-algebra and gives rise to a filtration on its specialization at $\hbar = c$, which can be identified with $\mathcal{R}_c^{-1} A$.
	\end{enumerate}
\end{proposition}
\begin{proof}
	For the first part, suppose that $r \in \mathcal{R}$ and $rx = 0$ (or $xr = 0$) for some $x \in A_\hbar$. For all $n \in \Z$, the degree $n$ component $x_n$ of $x$ satisfies $rx_n = 0$ (or $x_n r = 0$). Taking images in $A_{\hbar}|_{\hbar = c}$ gives $x_n \in (\hbar - c)$, but it is easily seen that $(\hbar - c)$ contains no homogeneous elements except zero since $c \in \Bbbk^{\times}$, hence $x_n = 0$ for all $n$.
	
	For the second part, we only discuss the case when $\mathcal{R}$ meets the right Ore condition. By an argument of rescaling, we assume $c = 1$ without loss of generality.
	
	For (i), the multiplicativity and right Ore condition for $\mathcal{R}_c$ can be checked within $A_{\hbar}$.
	
	For (ii), the elements of $\mathcal{R}^{-1} A_{\hbar}$ are expressed as formal fractions $xr^{-1}$. Since every $r \in \mathcal{R}$ is homogeneous, the required grading is evident. Moreover, if $\hbar \cdot xr^{-1} = 0$ then $\hbar xr' = 0$ for some $r' \in \mathcal{R}$, thus $xr' = 0$ and $xr^{-1} = 0$ in $\mathcal{R}^{-1} A_{\hbar}$. This implies $\mathcal{R}^{-1} A_{\hbar}$ is the Rees algebra of its specialization at $\hbar = 1$.
	
	The universal property for $A_{\hbar} \to \mathcal{R}^{-1} A_{\hbar}$ yields a homomorphism $\mathcal{R}^{-1} A_{\hbar} \to \mathcal{R}_c^{-1} A$, which is clearly surjective. It is routine to show that its kernel equals $(\hbar - c)$, whence the required isomorphism.
\end{proof}

\section{Bi-Whittaker reduction}\label{sec:bi-Whittaker}
\subsection{Quantum Toda lattices}\label{subsec:W}
Let $G$ be a connected reductive $\CC$-group. We fix a maximal unipotent subgroup $N$ of $G$. Fix a non-degenerate character $\bpsi: \mathfrak{n} \to \Ga$, i.e.\ $\bpsi(E) \neq 0$ for each simple root vector $E$. Specifically, we can and do take
\begin{itemize}
	\item a $G$-invariant symmetric non-degenerate pairing $(\cdot, \cdot)$ on $\mathfrak{g}$, thus $\mathfrak{g} \simeq \mathfrak{g}^*$;
	\item a Borel pair $(B, T)$ of $G$ such that $N := \mathcal{R}_{\mathrm{u}}(B)$;
	\item a principal $\mathfrak{sl}(2)$-triple $(\mathsf{e}, \mathsf{h}, \mathsf{f})$ in $\mathfrak{g}$ such that $\mathsf{e}$ (resp.\ $\mathsf{h}$, $\mathsf{f}$) lands in $\mathfrak{n}^-$ (resp.\ $\mathfrak{t}$, $\mathfrak{n}$), and
	\[ \bpsi(\xi) = (\mathsf{e}, \xi), \quad \xi \in \mathfrak{n}. \]
\end{itemize}

Let $\mathrm{W} := \mathrm{W}(G, T)$ denote the Weyl group.

The $\CC$-algebra $D(G)$ is positively filtered by the order filtration $D(G)_{\leq \bullet}$. Equip $\mathcal{U}\mathfrak{g}$ with the PBW filtration, so that the homomorphism $\mathcal{U}\mathfrak{g} \to D(G)$ is filtered. Denote by $\mathcal{Z}(\mathfrak{g})$ the center of $\mathcal{U}\mathfrak{g}$.

Consider the right $G \times G$-action on $G$ given by
\[ g \cdot (x, y) = x^{-1} g y, \]
and its infinitesimal versions. Hereafter, the subscripts $\Le$ and $\Ri$ will be affixed to $G$ or associated objects to indicate the ``left'' and ``right'' copies, respectively. For example, $G_{\Le} \times G_{\Ri}$ acts on $G$, and its composition with $\mathrm{diag}: G \to G_{\Le} \times G_{\Ri}$ is just the adjoint action.

The action above induces a homomorphism of algebras
\[ \tau: \mathcal{U} \mathfrak{g}_{\Le} \otimes \mathcal{U} \mathfrak{g}_{\Ri} \to D(G). \]
The endomorphism $\left[ \tau(\mathrm{diag}(\mathsf{h})), \cdot \right]$ of $D(G)$ is locally finite, hence gives rise to a $\Z$-grading, readily seen to be compatible with $D(G)_{\leq \bullet}$. The resulting Kazhdan filtration on $D(G)$ is denoted as $\mathrm{F}_\bullet D(G)$; it is not positively filtered in general.

Next, denote by $\mathfrak{n}_{\Le}^{\bpsi}$ and $\mathfrak{n}_{\Ri}^{\bpsi}$ the images of the maps
\[ \xi \mapsto (\xi - \bpsi(\xi)) \otimes 1, \quad \xi \mapsto 1 \otimes (\xi - \bpsi(\xi)), \]
respectively, where $\xi \in \mathfrak{n}$. They are commuting Lie subalgebras of $\mathcal{U}\mathfrak{g}_{\Le} \otimes \mathcal{U}\mathfrak{g}_{\Ri}$, and we adopt the shorthand
\[ \mathfrak{n}_{\LeRi}^{\bpsi} := \mathfrak{n}_{\Le}^{\bpsi} + \mathfrak{n}_{\Ri}^{\bpsi}. \]
The natural $N_{\Le} \times N_{\Ri}$-action on $D(G)$ preserves $\tau(\mathfrak{n}_{\LeRi}^{\bpsi})$. For brevity, we omit $\tau$ from the notation.

\begin{definition}
	With the notations above, set
	\[ \Wh := D(G) \myfatslash \mathfrak{n}_{\LeRi}^{\bpsi} := \left( D(G)/D(G)\mathfrak{n}_{\LeRi}^{\bpsi} \right)^{N_{\Le} \times N_{\Ri}}. \]
	This subquotient inherits the structure of algebra from $D(G)$, also known as the \emph{quantum Toda lattice} associated with $G$. We endow $\Wh$ with the filtration induced from $\mathrm{F}_\bullet D(G)$, denoted as $\mathrm{F}_\bullet \Wh$.
\end{definition}

Grade $\mathcal{U}\mathfrak{g}$ by the weights of $[\mathsf{h}, \cdot]$, and consider the resulting Kazhdan filtration. The natural homomorphism $\mathcal{Z}(\mathfrak{g}) \to D(G)$ induces an injection $\mathcal{Z}(\mathfrak{g}) \hookrightarrow \Wh$ of filtered algebras \cite[Section 1.2]{Gin18}.

In \textit{loc.\ cit.}, the group $G$ is assumed to be semisimple and simply connected. The construction of $\Wh$ can be understood as the \emph{bi-Whittaker} case of quantum Hamiltonian reduction, recalled below for later use.

\begin{definition}[\protect{\cite[Section 3.1]{Gin18}}]
	\label{def:quantum-Hamiltonian-reduction}
	Let $\mathfrak{h}$ be a reductive Lie algebra over $\CC$, and let $A$ be a finitely generated left Noetherian $\CC$-algebra endowed with an algebra homomorphism $\mathcal{U}\mathfrak{h} \to A$, such that the commutator $[x, \cdot]$ is a locally finite endomorphism on $A$ for all $x \in \mathfrak{h}$.
	
	Let $\mathfrak{k}$ be a nilpotent Lie subalgebra of $\mathcal{U}\mathfrak{h}$. Then $A/A\mathfrak{k}$ is an $A$-module on which $\mathfrak{k}$ acts locally nilpotently, and we define 
	\[ A \myfatslash \mathfrak{k} := (A/A\mathfrak{k})^{[\mathfrak{k}, \cdot]} \]
	where the superscript denotes the subspace annihilated by all $[x, \cdot]$ with $x \in \mathfrak{k}$.
\end{definition}

There is an analogous version with left and right swapped.

The following result extracted from \textit{loc.\ cit.} is also derived from this general framework.

\begin{proposition}\label{prop:W-hbar-Noetherian}
	The Rees $\CC[\hbar]$-algebra $\Wh_\hbar$ attached to $\mathrm{F}_\bullet \Wh$ (see Subsection \ref{subsec:fil-completions}) is Noetherian.
\end{proposition}
\begin{proof}
	The $\CC$-algebra $D(G)$ is finitely generated and Noetherian \cite[Proposition 1.4.6]{HTT08}. Denote by $D(G)_{\hbar}$ the Rees algebra associated with the Kazhdan filtration.

	By \cite[Lemma 3.3.2]{Gin18}, $\Wh_{\hbar}$ is isomorphic to $D(G)_\hbar \myfatslash \mathfrak{n}_{\LeRi}^{\bpsi}$. By the same result, $D(G)_\hbar$ is Noetherian, hence from \cite[Proposition 3.1.4 (iii)]{Gin18} we conclude that $D(G)_\hbar \myfatslash \mathfrak{n}_{\LeRi}^{\bpsi}$ is Noetherian.
\end{proof}

Consequently, $\Wh$, $\gr\Wh$ and the completion $\widehat{\Wh}$ are all Noetherian by Lemma \ref{prop:Ah-Noetherian}.

\subsection{Isogenies}\label{subsec:isogenies}
Retain the conventions of Subsection \ref{subsec:W} and let $A_G$ be the connected center of $G$. There is a canonical short exact sequence
\begin{equation}\label{eqn:G-SES}
	1 \to Z \to G' \xrightarrow{\pi} G \to 1, \quad G' := G_{\mathrm{SC}} \times A_G,
\end{equation}
where $Z$ is a finite central subgroup of $G'$. Equip $G'$ with the data $B'=T'N$ where $T' = \pi^{-1}(T)$, $(\mathsf{e}, \mathsf{h}, \mathsf{f})$, etc., so that $\pi$ preserves all these structures.

To $G$ and $G'$ are attached the quantum Toda lattices $\Wh_G$ and $\Wh_{G'}$, respectively. We have a canonical isomorphism
\begin{equation}\label{eqn:W-isogeny-prep}
	\Wh_{G'} \simeq \Wh_{G_{\mathrm{SC}}} \otimes D(A_G),
\end{equation}
compatibly with the homomorphisms from $\mathcal{Z}(\mathfrak{g}') \simeq \mathcal{Z}(\mathfrak{g}_{\mathrm{SC}}) \otimes \Sym \mathfrak{a}_G$; the right hand side of \eqref{eqn:W-isogeny-prep} is the tensor product of filtered algebras, where $D(A_G)$ carries the order filtration.

There is a natural $Z$-action on $D(G')$, namely the one induced by translation. The action passes to $\Wh_{G'}$ and preserves $\mathrm{F}_\bullet \Wh_{G'}$.

Since $\pi$ is finite étale, it induces a homomorphism of filtered algebras
\begin{equation}\label{eqn:W-isogeny}
	\Wh_G \simeq \Wh_{G'}^{Z\text{-inv}} \subset \Wh_{G'}.
\end{equation}

We record another fact related to translation for later use.

\begin{proposition}\label{prop:change-of-data}
	Given $G$, a non-degenerate invariant form $(\cdot, \cdot)$ on $\mathfrak{g}$ and principal $\mathfrak{sl}(2)$-triples $(\mathsf{e}_i, \mathsf{h}_i, \mathsf{f}_i)$ giving rise to maximal unipotent subgroups $N_i$ and non-degenerate characters $\bpsi_i: \mathfrak{n}_i \to \Ga$ for $i = 1, 2$, there is an isomorphism between the corresponding filtered algebras $\Wh_1 \rightiso \Wh_2$, which is canonical modulo $Z_G$-action.
\end{proposition}
\begin{proof}
	 There exists $g \in G$ such that $\Ad(g)(\mathsf{e}_1, \mathsf{h}_1, \mathsf{f}_1) = (\mathsf{e}_2, \mathsf{h}_2, \mathsf{f}_2)$; in particular $\Ad(g)N_1 = N_2$ and $\bpsi_2 \circ \Ad(g) = \bpsi_1$. Such a $g$ is unique up to $Z_G$. Left and right multiplication by $g$ induces the desired isomorphism.
\end{proof}

The datum $(\cdot, \cdot)$ above is immaterial, since it is unique up to $\CC^{\times}$ on each simple factor of $\mathfrak{g}$.

\subsection{Bi-invariant functions}
Denote by $w_0$ the longest element of $\mathrm{W}$, and let $\rho \in \mathbf{X}^*(T)$ (resp.\ $\check{\rho} \in \mathbf{X}_*(T)$) be the half-sum of $B$-positive roots (resp.\ coroots). Denote by $\mathbf{X}^*(T)^+$ the subset of $B$-dominant elements in $\mathbf{X}^*(T)$. Let $V_\lambda$ be a simple $G$-module of $B$-highest weight $\lambda \in \mathbf{X}^*(T)^+$, then $V_\lambda^* \simeq V_{-w_0 \lambda}$.

In view of the algebraic Peter--Weyl theorem, taking matrix coefficients yields
\begin{equation}\label{eqn:bi-invariant-sum}
	\CC[G]^{N_{\Le} \times N_{\Ri}} \simeq \bigoplus_{\lambda \in \mathbf{X}^*(T)^+} \left( V_{-w_0 \lambda}^N \otimes V_\lambda^N \right).
\end{equation}

\begin{lemma}
	The subalgebra $\CC[G]^{N_{\Le} \times N_{\Ri}}$ of $D(G)$ injects into $\Wh = D(G) \myfatslash \mathfrak{n}_{\LeRi}^{\bpsi}$.
\end{lemma}
\begin{proof}
	In fact, $\CC[G] \cap D(G) \mathfrak{n}_{\LeRi}^{\bpsi} = \{0\}$ inside $D(G)$. This is clear by consideration of principal symbols of differential operators restricted to $Bw_0 B$.
\end{proof}

The adjoint $\mathsf{h}$-action on $D(G)$ also leaves $\CC[G]^{N_{\Le} \times N_{\Ri}}$ invariant. The weights can thus be read from the isomorphism above.

\begin{lemma}\label{prop:bi-invariant-grading}
	For each nonzero $v^* \otimes v$ in the summand indexed by $\lambda \in \mathbf{X}^*(T)^+$, its weight under the adjoint $\mathsf{h}$-action is $-\lrangle{\lambda, 4\check{\rho}} \in \Z_{\leq 0}$.
	
	As a consequence, for all $n \in \Z$, the image of $v^* \otimes v$ in $\CC[G]^{N_{\Le} \times N_{\Ri}}$ belongs to $\mathrm{F}_n \Wh$ if and only if $n \geq -\lrangle{\lambda, 4\check{\rho}}$.
\end{lemma}
\begin{proof}
	For the first part, the $\mathsf{h}$-action is just that of $\mathsf{h} \otimes 1 + 1 \otimes \mathsf{h}$. It remains to note that $w_0 \mathsf{h} = -\mathsf{h}$, and $\mathsf{h}$ is the derivative of $-2\check{\rho}: \Gm \to T$ (see eg.\ \cite[(2.3)]{Gr97}). The negative sign is present because the $\mathsf{e}$ in the $\mathfrak{sl}(2)$-triple lands in $\mathfrak{n}^-$ .
	
	The second part is then clear by definition of Kazhdan filtrations.
\end{proof}

\begin{definition}
	Let $\mathcal{I}^\natural \subset \CC[G]$ be the defining ideal of $G \smallsetminus Bw_0 B$, and put
	\[ \mathcal{I} := \mathcal{I}^\natural \cap \CC[G]^{N_{\Le} \times N_{\Ri}}. \]
\end{definition}

\begin{lemma}\label{prop:d}
	Take $\lambda \in \mathbf{X}^*(T)$ such that $\lrangle{\lambda, \check{\alpha}} > 0$ for all $B$-positive roots $\alpha$. Let $d$ be any generator of the direct summand of $\CC[G]^{N_{\Le} \times N_{\Ri}}$ indexed by $\lambda$. Then $d \in \mathcal{I}^\natural$, and $\sqrt{(d)} = \mathcal{I}^{\natural}$.
\end{lemma}
\begin{proof}
	This is well known. A proof can be found in \cite{EMB20}.
\end{proof}

Consequently, $\sqrt{(d)} = \mathcal{I}$ also holds in $\CC[G]^{N_{\Le} \times N_{\Ri}}$.

As a by-product of the foregoing results, we now show that $\Wh$ is not a Zariskian filtered ring by leveraging ideas from harmonic analysis on real reductive groups.

\begin{proposition}\label{prop:non-Zariskian}
	If $G$ is not a torus, then the filtration $\mathrm{F}_\bullet \Wh$ on the ring $\Wh$ is neither faithful nor Zariskian (Remark \ref{rem:Zariskian}).
\end{proposition}
\begin{proof}
	In view of Proposition \ref{prop:change-of-data}, we may assume the data $G \supset B = TN$ and $(\mathsf{e}, \mathsf{h}, \mathsf{f})$ in Subsection \ref{subsec:W} are all defined over $\R$, whereas $(\cdot, \cdot)$ is $\sqrt{-1}$ times an invariant form defined over $\R$. In particular, $\bpsi$ is imaginary-valued on $\mathfrak{n}(\R)$.
	
	Since $\Wh_{\hbar}$ is known to be left Noetherian (Proposition \ref{prop:W-hbar-Noetherian}), by the last part of \cite[2.7 Proposition]{ABO89} it suffices to show $\mathrm{F}_\bullet \Wh$ is not faithful. By (1) of \textit{loc.\ cit.}, it suffices to show $\mathrm{F}_{-1} \Wh$ is not included in the Jacobson radical of $\mathrm{F}_0 \Wh$.
	
	Take $d$ as in Lemma \ref{prop:d}, then $d \in \mathrm{F}_{-1} \Wh$ by Lemma \ref{prop:bi-invariant-grading} since $G \neq T$; by rescaling, we may assume $d(\dot{w}_0) = 1$ where $\dot{w}_0$ is a representative of $w_0$. We claim that $1 - d \notin \Wh^{\times}$; in turn, this will imply that $d$ is not in the Jacobson radical of $\mathrm{F}_0 \Wh$.
	
	Denote by $\Schw^*(G(\R))$ the continuous dual of the Fréchet space $\Schw(G(\R))$ of Schwartz measures on the complex group $G(\R)$, i.e.\ the space of tempered distributions on $G(\R)$; the elements of $\Schw(G(\R))$ are viewed as functions by fixing a Haar measure. Note that $D(G)$ is defined over $\R$ and acts on $\Schw^*(G(\R))$, hence $\mathcal{U}\mathfrak{g}_{\Le} \otimes \mathcal{U}\mathfrak{g}_{\Ri}$ acts on $\Schw^*(G(\R))$ as well. Define the subspace
	\[ \Schw^*(G(\R))^{\mathfrak{n}_{\LeRi}^{\bpsi}} := \left\{ \Lambda \in \Schw^*(G(\R)) : \mathfrak{n}_{\LeRi}^{\bpsi} \Lambda = 0  \right\}. \]
	
	There is an induced $\Wh$-action on $\Schw^*(G(\R))^{\mathfrak{n}_{\LeRi}^{\bpsi}}$. In order to prove the claim, it remains to exhibit $\Lambda \in \Schw^*(G(\R))^{\mathfrak{n}_{\LeRi}^{\bpsi}} \smallsetminus \{0\}$ such that $d\Lambda = \Lambda$.
	
	Fix a Haar measure on $N(\R)$. By our assumptions, $\bpsi$ integrates to a unitary character $\Psi$ of $N(\R)$. Define
	\[ \Lambda: \Schw(G(\R)) \to \CC, \quad \Lambda(f) = \int_{N(\R) \times N(\R)} \Psi(n_1) f\left( n_1^{-1} \dot{w}_0 n_2 \right) \Psi(n_2)^{-1} \dd n_1 \dd n_2. \]
	A result of Rosenlicht implies that all $N_{\Le}(\R) \times N_{\Ri}(\R)$-orbits on $G(\R)$ are closed, hence $\Lambda \in \Schw^*(G(\R))$. Clearly $\mathfrak{n}_{\LeRi}^{\bpsi} \Lambda = 0$. On the other hand, $d\Lambda = \Lambda$ since $d(n_1 \dot{w}_0 n_2) = d(\dot{w}_0) = 1$ for all $n_1, n_2 \in N(\R)$. This completes the proof.
\end{proof}

\section{Universal centralizer}\label{sec:universal-centralizer}
\subsection{Review of constructions}
Let $G$ be a connected reductive $\CC$-group. The universal centralizer of $G$ is a smooth $\CC$-variety $\mathfrak{Z} = \mathfrak{Z}_G$ equipped with a symplectic structure and a characteristic morphism $\kappa$ to $\mathfrak{c} := \mathfrak{g}^* \sslash_{\Ad^*} G$. To fix notations, we summarize the constructions below.

The first construction goes as follows. Set
\[ \mathscr{Z} := \left\{ (x, g) \in \mathfrak{g}^* \times G : \Ad^*(g)(x) = x \right\} \]
and let $G$ act diagonally on $\mathfrak{Z}$.

Fix $(\cdot, \cdot)$ as in Subsection \ref{subsec:W} to identify $\mathfrak{g}$ and $\mathfrak{g}^*$. Denote by $\mathfrak{g}_{\mathrm{reg}} \subset \mathfrak{g}$ the regular (not necessarily semisimple) locus, and denote its image in $\mathfrak{g}^*$ by $\mathfrak{g}^*_{\mathrm{reg}}$, which is independent of $(\cdot, \cdot)$. Define
\[ \mathscr{Z}_{\mathrm{reg}} := \left\{ (x, g) \in \mathfrak{Z}: x \in \mathfrak{g}^*_{\mathrm{reg}} \right\}. \]
This is a commutative $\mathfrak{g}^*_{\mathrm{reg}}$-group scheme.

\begin{definition-proposition}[\protect{\cite[Proposition 3.3.9]{Ri17}}]
	\label{def:universal-centralizer}
	There exists a unique smooth affine commutative $\mathfrak{c}$-group scheme $\mathfrak{Z}$ whose pull-back via $\mathfrak{g}^*_{\mathrm{reg}} \to \mathfrak{c}$ is $\mathscr{Z}_{\mathrm{reg}}$.
	
	We call $\mathfrak{Z}$ the \emph{universal centralizer} of $G$, and $\kappa: \mathfrak{Z} \to \mathfrak{c}$ its \emph{characteristic morphism}.
\end{definition-proposition}

In order to better understand $\mathfrak{Z}$, we fix a Borel pair $B = TN$ and an adapted principal $\mathfrak{sl}(2)$-triple $(\mathsf{e}, \mathsf{h}, \mathsf{f})$ as in Subsection \ref{subsec:W}. Let $\mathfrak{g}_{\mathsf{f}}$ be the centralizer subalgebra of $\mathsf{f}$, and define
\[ \mathscr{S} := \mathsf{e} + \mathfrak{g}_{\mathsf{f}}. \]
For convenience, we will often identify $\mathscr{S}$ with its image under $\mathfrak{g} \simeq \mathfrak{g}^*$.

We will use results about Kostant slices freely. First off, $\mathscr{S} \subset \mathsf{e} + \mathfrak{b} \subset \mathfrak{g}_{\mathrm{reg}}$; moreover $\mathfrak{g}^* \to \mathfrak{c}$ restricts to an isomorphism $\mathscr{S} \rightiso \mathfrak{c}$.

Next, define
\[ \mathscr{Z}_{\mathscr{S}} := \left\{ (x, g) \in \mathscr{Z}: x \in \mathscr{S} \right\} \subset \mathscr{Z}_{\mathrm{reg}}. \]
This is a group scheme over $\mathscr{S} \simeq \mathfrak{c}$.

\begin{remark}\label{rem:Z-dim}
	The isomorphism above leads easily to the standard fact that $\dim \mathfrak{Z} = 2\dim T$.
\end{remark}

\begin{proposition}[\protect{\cite[Remark 3.3.10]{Ri17}}]
	\label{prop:universal-centralizer-slice}
	There is a natural isomorphism of $\mathfrak{c}$-group schemes $\mathscr{Z}_{\mathscr{S}} \simeq \mathfrak{Z}$.
\end{proposition}

The description of $\mathscr{Z}_{\mathrm{reg}}$ as the pull-back via $\mathfrak{g}^*_{\mathrm{reg}} \to \mathfrak{c}$ implies
\begin{equation}\label{eqn:universal-centralizer-catquot}
	\mathfrak{Z} \simeq \mathscr{Z}_{\mathrm{reg}} \sslash G,
\end{equation}
so that $\kappa$ is induced by the first projection $\pi_1: \mathscr{Z}_{\mathrm{reg}} \to \mathfrak{g}^*$.

There is also a canonical symplectic structure on $\mathfrak{Z}$ given as follows. We have
\[ \mathrm{T}^* G = \left\{ (x', g, x) \in \mathfrak{g}^* \times G \times \mathfrak{g}^*: x' = \Ad^*(g)(x) \right\}. \]
Let $G_{\Le} \times G_{\Ri}$ act on the right of $G$, the moment map is then $\bm{\mu}_{\Le} \times \bm{\mu}_{\Ri}: \mathrm{T}^* G \to \mathfrak{g}^* \times \mathfrak{g}^*$ where
\[ \bm{\mu}_{\Le}(x', g, x) = x', \quad \bm{\mu}_{\Ri}(x', g, x) = x. \]
The adjoint action of $G$ has moment map given by $\bm{\mu}_{\Ad}(x', g, x) = x' - x \in \mathfrak{g}^*$, so $\bm{\mu}_{\Ad}^{-1}(0) = \mathscr{Z}$.

Denote by $\mathrm{T}^*_{\mathrm{reg}} G \subset \mathrm{T}^* G$ the open locus where $x$ (equivalently $x'$) lies in $\mathfrak{g}^*_{\mathrm{reg}}$. In view of \eqref{eqn:universal-centralizer-catquot}, now $\mathfrak{Z}$ can be interpreted as the Hamiltonian reduction
\begin{equation}\label{eqn:Z-symplectic}
	\mathfrak{Z} = \left( \bm{\mu}_{\Ad}^{-1}(0) \cap \mathrm{T}^*_{\mathrm{reg}} G \right) \sslash_{\Ad} G.
\end{equation}
Therefore, $\mathfrak{Z}$ receives the symplectic structure from $\mathrm{T}^* G$, independent of all choices. We refer to \cite[Section 3.3]{Ri17} or \cite[Section 2.1]{Gin18} for more details about this construction.

The second construction of $\mathfrak{Z}$ realizes it as a bi-Whittaker Hamiltonian reduction, whose general definition can be found in \cite[Section 2.2]{Gin18}. Below is a recap.

A key property of Kostant slices is the isomorphism
\begin{align*}
	N \times \mathscr{S} & \rightiso \mathsf{e} + \mathfrak{b} \\
	(n, y) & \mapsto \Ad(n)(y)
\end{align*}
of varieties. Recall that $\bpsi = (\mathsf{e}, \cdot)$. By identifying $\mathscr{S}$ with its image in $\mathfrak{g}^*$, we see $\Ad^*$ induces
\begin{equation*}
	N \times \mathscr{S} \rightiso \bpsi + \mathfrak{n}^\perp.
\end{equation*}
In particular, $\bpsi + \mathfrak{n}^\perp$ is $N$-invariant, and
\begin{equation}\label{eqn:Z-2}\begin{aligned}
	\mathrm{T}^* G \myfatslash (N_{\Le} \times N_{\Ri}, \bpsi \times \bpsi) & := \left( \bm{\mu}_{\Le}^{-1}(\bpsi + \mathfrak{n}^\perp) \cap \bm{\mu}_{\Ri}^{-1}(\bpsi + \mathfrak{n}^\perp) \right) \big/ N_{\Le} \times N_{\Ri} \\
	& \rightiso \bm{\mu}_{\Le}^{-1}(\mathscr{S}) \cap \bm{\mu}_{\Ri}^{-1}(\mathscr{S}) \quad \text{(by restriction)} \\
	& \simeq \left\{ (x, g) \in \mathscr{S} \times G : \Ad^*(g)(x) \in \mathscr{S} \right\} \quad \text{(description of $\bm{\mu}_{\Le}, \bm{\mu}_{\Ri}$)} \\
	& = \left\{ (x, g) \in \mathscr{S} \times G : \Ad^*(g)(x) = x \right\} \quad \text{(property of $\mathscr{S}$)} \\
	& = \mathscr{Z}_{\mathscr{S}} \simeq \mathfrak{Z} \quad \text{(by Proposition \ref{prop:universal-centralizer-slice})}.
\end{aligned}\end{equation}

\subsection{Bi-invariant functions and Bruhat cells}
Recall that $\bpsi + \mathfrak{n}^\perp$ is $N$-invariant under $\Ad^*$. In \eqref{eqn:Z-2} we considered
\[ \mathrm{T}^* G \myfatslash (N_{\Le} \times N_{\Ri}, \bpsi \times \bpsi) := \left( \bm{\mu}_{\Le}^{-1}(\bpsi + \mathfrak{n}^\perp) \cap \bm{\mu}_{\Ri}^{-1}(\bpsi + \mathfrak{n}^\perp) \right) \big/ N_{\Le} \times N_{\Ri}. \]
Each element of $\CC[G]^{N_{\Le} \times N_{\Ri}}$ pulls back to a regular function on the right-hand side. This yields a homomorphism of algebras
\begin{equation}\label{eqn:iota-B}
	\iota: \CC[G]^{N_{\Le} \times N_{\Ri}} \to \CC \left[ \mathrm{T}^* G \myfatslash (N_{\Le} \times N_{\Ri}, \bpsi \times \bpsi) \right] \rightiso \CC[\mathfrak{Z}].
\end{equation}
The first arrow is the restriction of $\CC[G] \to \CC[\mathrm{T}^* G]$, hence injective. Also, the first arrow depends only on $N$ and $\bpsi$, whereas the second depends on $(\mathsf{e}, \mathsf{h}, \mathsf{f})$ and $(\cdot, \cdot)$ as well.

Let $\pi_2: \mathfrak{Z} \simeq \mathscr{Z}_{\mathscr{S}} \to G$ be the second projection $(x, g) \mapsto g$. The resulting set-theoretic map $\mathfrak{Z} \to N \backslash G / N$ depends only on $N$ and $\bpsi$.

\begin{definition}
	For each $w \in \mathrm{W}(G, T)$, define the corresponding cell $\mathfrak{B}_w := \pi_2^{-1}(BwB)$. They depend on the choice of $(\mathsf{e}, \mathsf{h}, \mathsf{f})$.
\end{definition}

The cells $\mathfrak{B}_w$ stratify $\mathfrak{Z}$ into locally closed subsets since so do $BwB$ for $G$, and the \emph{big cell} $\mathfrak{B}_{w_0}$ is open. The following is a more precise statement.

\begin{lemma}\label{prop:B-d}
	Pick $d \in \CC[G]^{N_{\Le} \times N_{\Ri}} \smallsetminus \{0\}$ as in Lemma \ref{prop:d}, then $\mathfrak{B}_{w_0}$ is the principal open subset defined by $\iota(d)$.
\end{lemma}
\begin{proof}
	Given $(x, g) \in \mathscr{Z}_{\mathscr{S}}$, we have $g \in Bw_0 B$ if and only if $d(g) \neq 0$, if and only if the pull back of $d$ to $\mathrm{T}^* G$ is nonzero at $(x, g) \in \mathscr{Z}_{\mathscr{S}} \hookrightarrow \mathrm{T}^* G$.
\end{proof}

To describe $\mathfrak{B}_{w_0}$, we pick a representative $\tilde{w}_0$ of $w_0$, and identify
\[ \mathfrak{Z} \simeq \mathrm{T}^* G \myfatslash (N_{\Le} \times N_{\Ri}, \bpsi \times \bpsi). \]
We also identify $\mathfrak{g}$ with $\mathfrak{g}^*$ via $(\cdot, \cdot)$, so that $\bpsi + \mathfrak{n}^\perp = \mathsf{e} + \mathfrak{b}$.

\begin{lemma}\label{prop:big-cell}
	Up to $N_{\Le} \times N_{\Ri}$-action, the elements of $\mathfrak{B}_{w_0}$ are of the form
	\[ ( \underbracket{\mathsf{e} + t + \Ad(\tilde{w}_0^{-1} h)^{-1}(\mathsf{e})}_{\in \mathsf{e} + \mathfrak{b}}, \; \tilde{w}_0^{-1} h ), \quad h \in T, \; t \in \mathfrak{t}. \]
\end{lemma}
\begin{proof}
	This is \cite[Example 2.4]{Jin22}, except that the roles of $\mathsf{e}$ and $\mathsf{f}$ are swapped there.
\end{proof}

For the case $G = T$ we have $ \mathfrak{B}_{w_0} = \mathfrak{Z} = T \times \mathfrak{t}^*$ with $\kappa$ being the second projection. For groups of the form $G = G_1 \times G_2$, we also have $\mathfrak{Z} = \mathfrak{Z}_1 \times \mathfrak{Z}_2$ compatibly with characteristic morphisms and Bruhat cells. We omit these trivialities.

\subsection{Isogenies}
The following material can also be found in \cite[Section 2]{BFM05}.

Let $G'$ be a connected reductive $\CC$-group, $Z$ a finite central subgroup of $G'$, and $G := G'/Z$. Let $\mathscr{Z}'$ (resp.\ $\mathscr{Z}$) denote the variety defined earlier for $G'$ (resp.\ $G$). Then $Z$ acts on $\mathscr{Z}'$ by $(x, g) \mapsto (x, zg)$, and there is a $Z$-torsor $\mathscr{Z}' \to \mathscr{Z}$ given by $(x, g) \mapsto (x, gZ)$.

Write $\mathfrak{Z} = \mathfrak{Z}_G$ and $\mathfrak{Z}' = \mathfrak{Z}_{G'}$. Observe that $\mathfrak{c} = \mathfrak{c}'$.

\begin{proposition}\label{prop:isogeny-Z}
	The morphism $\mathscr{Z}' \to \mathscr{Z}$ induces $\mathfrak{Z}' \to \mathfrak{Z}$ that is a $Z$-torsor preserving the characteristic morphisms to $\mathfrak{c}$ and symplectic structures.
\end{proposition}
\begin{proof}
	One readily checks that $\mathscr{Z}' \to \mathscr{Z}$ is compatible with diagonal actions of $G'$ and $G$; it also induces $\mathscr{Z}'_{\mathrm{reg}} \to \mathscr{Z}_{\mathrm{reg}}$. Taking categorical quotients, we get the $Z$-torsor $\mathfrak{Z}' \to \mathfrak{Z}$ that commutes with characteristic morphisms $\kappa': \mathfrak{Z}' \to \mathfrak{c}'$ and $\kappa: \mathfrak{Z} \to \mathfrak{c}$.
	
	As for symplectic structures, it suffices to gaze at \eqref{eqn:Z-symplectic}, noting that the finite étale $G' \to G$ induces $\mathrm{T}^*(G') \to \mathrm{T}^*(G)$ which preserves regular loci and symplectic forms. 
\end{proof}

Choosing $(\mathsf{e}, \mathsf{h}, \mathsf{f})$ and the Borel pairs $(B', T')$, $(B, T)$ compatibly with $G' \to G$, one readily sees $\mathscr{S} = \mathscr{S}'$. Furthermore, the construction of $\mathfrak{Z}$ via bi-Whittaker Hamiltonian reduction is also compatible with $G' \to G$.

\begin{lemma}\label{prop:isogeny-cell}
	With the choices above, there is a commutative diagram
	\[\begin{tikzcd}
		\CC[\mathfrak{Z}] \arrow[r] & \CC[\mathfrak{Z}'] \\
		\CC[G]^{N_{\Le} \times N_{\Ri}} \arrow[hookrightarrow, u, "{\iota}"] \arrow[r] & \CC[G']^{N_{\Le} \times N_{\Ri}} \arrow[hookrightarrow, u, "{\iota'}"'] 
	\end{tikzcd}\]
	where the horizontal arrows are pull backs along $\mathfrak{Z}' \to \mathfrak{Z}$ and $G' \to G$ respectively, and the vertical arrows are as in \eqref{eqn:iota-B}.
	
	Furthermore, $\mathfrak{B}'_w$ is a $Z$-torsor over $\mathfrak{B}_w$ for all $w \in \mathrm{W}(G', T') = \mathrm{W}(G, T)$.
\end{lemma}
\begin{proof}
	Clear from the preceding discussions.
\end{proof}

\subsection{Intersection of fibers with the big cell}
Fix a connected reductive $\CC$-group $G$, $B = TN$ and $(\mathsf{e}, \mathsf{h}, \mathsf{f})$ as before, so that the cells $\mathfrak{B}_w \subset \mathfrak{Z}$ are defined. We use the previous result to prove the following fact.

\begin{lemma}\label{prop:fiber-cell}
	Let $\nu$ be any point of $\mathfrak{c}$. Then $\mathfrak{B}_{w_0}$ intersects all connected components of $\kappa^{-1}(\nu)$.
\end{lemma}
\begin{proof}
	We begin with the case when $G$ is adjoint. The fiber $\kappa^{-1}(\nu)$ is isomorphic to the centralizer of the element $x \in \mathscr{S}$ corresponding to $\nu$, thus connected by \cite[Proposition 2.4]{Ko79}.
	
	With the notations of Proposition \ref{prop:big-cell}, elements of $\mathfrak{B}_{w_0}$ may be represented by $(\mathsf{g} + t, \tilde{w}_0^{-1} h)$ where $\mathsf{g} := \mathsf{e} + \Ad(\tilde{w}_0^{-1} h)^{-1}(\mathsf{e})$; its image under $\kappa$ is the image of $\mathsf{g} + t$ under $\mathfrak{g}^* \to \mathfrak{c}$. Fix $h \in T$. It suffices to show that
	\begin{equation}\label{eqn:fiber-cell-aux}
		\exists t \in \Lie(T), \; \mathsf{g} + t \mapsto \nu.
	\end{equation}

	We may assume $(\mathsf{e}, \mathsf{h}, \mathsf{f})$ arises from a pinning $(B^-, T, (E_\alpha)_\alpha)$ of $G$ where $\alpha$ ranges over the $B^-$-simple roots and $\mathfrak{g}_\alpha = \CC E_\alpha$, so that
	\[ \mathsf{e} = \sum_{\alpha} E_\alpha, \]
	see eg. \cite[Section 2]{Gr97}. Identify $\mathfrak{c}$ with $\CC^{\mathrm{rk}(G)}$ using the fundamental invariants in $\CC[\mathfrak{g}]^G$. Then:
	\begin{itemize}
		\item $\Ad(\tilde{w}_0^{-1} h)^{-1}(\mathsf{e}) = \sum_\alpha t_\alpha E_\alpha$ where $\alpha$ ranges over the $B$-simple roots and $t_\alpha \in \CC$;
		\item $\Ad(\tilde{w}_0^{-1} h)^{-1}(\mathsf{e})$ belongs to the subspace $d_1 \subset \mathfrak{g}$ of \cite[(2.2.1)]{Ko79}, thus our $\mathsf{g}$ meets the requirements in \cite[Proposition 2.5.1]{Ko79}; note that our $\mathsf{e}$ matches the $\mathsf{f}$ in \text{loc.\ cit.};
		\item the map $\mathscr{J}$ in \cite[(2.3.1)]{Ko79} is exactly the quotient morphism $\mathfrak{g} \to \mathfrak{c}$.
	\end{itemize}
	Therefore \cite[Proposition 2.5.1]{Ko79} yields \eqref{eqn:fiber-cell-aux}; here our $t$ corresponds to the $x$ in \textit{loc.\ cit.}
	
	Next, suppose $G$ is semisimple. Take $\underline{G} = G/Z_G$ and denote by $\underline{\kappa}: \underline{\mathfrak{Z}} \to \mathfrak{c}$ the characteristic morphism for $\underline{G}$. Proposition \ref{prop:isogeny-Z} implies that $\kappa^{-1}(\nu)$ is the preimage of $\underline{\kappa}^{-1}(\nu)$ under the $Z_G$-torsor $\mathfrak{Z} \to \underline{\mathfrak{Z}}$.
	
	By general properties of coverings, $Z_G$ acts on $\pi_0(\kappa^{-1}(\nu))$ transitively, whereas $\mathfrak{B}_{w_0}$ is the preimage of $\underline{\mathfrak{B}}_{w_0}$ by Lemma \ref{prop:isogeny-cell}. The adjoint case implies that some connected component of $\kappa^{-1}(\nu)$ meets $\mathfrak{B}_{w_0}$; by the transitive action of $Z_G$, every connected component meets $\mathfrak{B}_{w_0}$.
	
	The case where $G$ is a torus is trivial, and so is the case when $G$ is a direct product of a torus and a semisimple group. For general $G$, consider the isogeny $G \to G_{\mathrm{ab}} \times (G/Z_G)$ and reason as above.
\end{proof}

\begin{remark}
	The result above is also implicit in \cite[Section 5.2]{Jin22}. The author is grateful to Xin Jin for pointing this out.
\end{remark}

\section{Properties of \texorpdfstring{$\Wh$}{W}}\label{sec:W-properties}
Throughout this section, $G$ will be a connected reductive $\CC$-group. Fix the data in Subsection \ref{subsec:W} to define $\Wh = \Wh_G$.

\subsection{Filtration and grading}
We summarize certain results from \cite{Gin18}, extended to the reductive setting.

\begin{theorem}\label{prop:separating}
	The filtration $\mathrm{F}_\bullet \Wh$ is separating.
\end{theorem}
\begin{proof}
	For simply connected $G$, this is \cite[Corollary 4.3.2]{Gin18}. The general case follows easily from \eqref{eqn:W-isogeny-prep} and \eqref{eqn:W-isogeny}.
\end{proof}

\begin{theorem}\label{prop:grW-Z}
	There is an isomorphism of graded Poisson algebras $\gr \Wh \simeq \CC[\mathfrak{Z}]$ that restricts to an isomorphism of maximal commutative subalgebras $\gr\mathcal{Z}(\mathfrak{g}) \simeq \kappa^*(\CC[\mathfrak{c}])$, where $\mathcal{Z}(\mathfrak{g})$ carries the filtration induced from $\mathrm{F}_\bullet \Wh$; note that $\kappa^*$ is injective.
\end{theorem}
\begin{proof}
	For simply connected $G$, this is \cite[Theorem 1.2.2]{Gin18}, the general case follows from \eqref{eqn:W-isogeny-prep}, \eqref{eqn:W-isogeny} on the side of $\Wh$, and Proposition \ref{prop:isogeny-Z} on the side of $\mathfrak{Z}$.
\end{proof}

\begin{proposition}\label{prop:PBW-center}
	The filtration above on $\mathcal{Z}(\mathfrak{g})$ is the one induced from the PBW filtration on $\mathcal{U}(\mathfrak{g})$. In particular, it is positively filtered.
\end{proposition}
\begin{proof}
	The adjoint $\mathsf{h}$-weights on $\mathcal{Z}(\mathfrak{g})$ must be zero.
\end{proof}

We require another compatibility result to be used with Theorem \ref{prop:grW-Z}. Recall that $\CC[G]^{N_{\Le} \times N_{\Ri}}$ embeds into $\Wh$ and inherits the induced filtration.

\begin{lemma}\label{prop:bi-invariant-filtration}
	The filtration on $\CC[G]^{N_{\Le} \times N_{\Ri}}$ induced from $\mathrm{F}_\bullet \Wh$ comes from a canonical $2\Z_{\leq 0}$-grading.
\end{lemma}
\begin{proof}
	Combine \eqref{eqn:bi-invariant-sum} and Lemma \ref{prop:bi-invariant-grading} to obtain the required grading indexed by $2\Z_{\leq 0}$.
\end{proof}

On the other hand, \eqref{eqn:iota-B} gives an embedding $\iota: \CC[G]^{N_{\Le} \times N_{\Ri}} \hookrightarrow \CC[\mathfrak{Z}]$ of algebras.

\begin{proposition}
	There is a commutative diagram
	\[\begin{tikzcd}[column sep=large, row sep=large]
		\CC[\mathfrak{Z}] \arrow[r, "\sim"', "\text{Theorem \ref{prop:grW-Z}}"] & \gr\Wh \\
		\CC[G]^{N_{\Le} \times N_{\Ri}} \arrow[hookrightarrow, u, "\iota"] \arrow[r, "\sim", "\text{Lemma \ref{prop:bi-invariant-filtration}}"'] & \gr \left(\CC[G]^{N_{\Le} \times N_{\Ri}}\right). \arrow[hookrightarrow, u]
	\end{tikzcd}\]
\end{proposition}
\begin{proof}
	We have to examine the proof of \cite[Theorem 1.2.2]{Gin18} in Section 4.3 of \textit{loc.\ cit.} The isomorphism $\CC[\mathfrak{Z}] \rightiso \gr^{\mathrm{F}} \Wh$ (resp.\ $\CC[G]^{N_{\Le} \times N_{\Ri}} \rightiso \gr^{\mathrm{F}} \left( \CC[G]^{N_{\Le} \times N_{\Ri}} \right)$) is the composite of the first (resp.\ second) column of the diagram
	\[\begin{tikzcd}
		\CC[\mathfrak{Z}] \arrow[d, "\sim" sloped, "\alpha"'] & \CC[G]^{N_{\Le} \times N_{\Ri}} \arrow[hookrightarrow, l] \arrow[equal, dd] \\
		\CC \left[ \mathrm{T}^* G \myfatslash (N_{\Le} \times N_{\Ri}, \bpsi \times \bpsi) \right] \arrow[d, "\sim" sloped, "\beta"'] & \\
		\left( \gr^{\leq} D(G) \right) \myfatslash (N_{\Le} \times N_{\Ri}, \bpsi \times \bpsi) \arrow[d, "\sim" sloped, "\gamma"'] & \CC[G]^{N_{\Le} \times N_{\Ri}} \arrow[hookrightarrow, l] \arrow[d, "\sim" sloped, "{\gamma'}"'] \\
		\left( \gr^{\mathrm{F}} D(G) \right) \myfatslash (N_{\Le} \times N_{\Ri}, \bpsi \times \bpsi) \arrow[d, "\sim" sloped, "\delta"'] & \gr^{\mathrm{F}} \CC[G]^{N_{\Le} \times N_{\Ri}} \arrow[hookrightarrow, l] \arrow[equal, d] \\
		\gr^{\mathrm{F}} \left( D(G) \myfatslash (N_{\Le} \times N_{\Ri}, \bpsi \times \bpsi) \right) \arrow[equal, d] & \gr^{\mathrm{F}} \CC[G]^{N_{\Le} \times N_{\Ri}} \arrow[hookrightarrow, l] \\
		\gr^{\mathrm{F}} \Wh. &
	\end{tikzcd}\]
	The notation $\myfatslash$ from Definition \ref{def:quantum-Hamiltonian-reduction} and \eqref{eqn:Z-2} remains in effect here. The horizontal inclusions are the evident ones, and the vertical arrows are explained below:
	\begin{itemize}
		\item $\alpha$ is the construction of $\mathfrak{Z}$ via bi-Whittaker quantum Hamiltonian reduction;
		\item $\beta$ is the composite of
		\begin{itemize}
			\item $\gr^{\leq} D(G) \simeq \CC[\mathrm{T}^* G]$, and
			\item $\CC \left[ \mathrm{T}^* G \myfatslash (N_{\Le} \times N_{\Ri}, \bpsi \times \bpsi) \right] = \CC[\mathrm{T}^* G] \myfatslash (N_{\Le} \times N_{\Ri}, \bpsi \times \bpsi)$, which is implicit in \cite[Section 4.3]{Gin18} and follows easily from Kostant slices;
		\end{itemize}
		\item $\gamma$ is the ``regrading'' isomorphism (Proposition \ref{prop:regrading});
		\item $\gamma'$ is a similar regrading, placing the $\lambda$-summand in \eqref{eqn:bi-invariant-sum} at grade $-\lrangle{\lambda, 4\check{\rho}}$ (cf.\ Lemma \ref{prop:bi-invariant-grading});
		\item $\delta$ is the second isomorphism in \cite[Lemma 3.3.2 (1)]{Gin18}, induced from the evident surjection
		\begin{multline*}
			\mathrm{F}_n D(G) \big/ \sum_i (\mathrm{F}_{n-1-i} D(G)) \mathfrak{n}_{\LeRi, i}^{\bpsi} \twoheadrightarrow \\
			\left(\mathrm{F}_n D(G) + D(G)\mathfrak{n}_{\LeRi}^{\bpsi}\right) \big/ \left( \mathrm{F}_{n-1} D(G) + D(G)\mathfrak{n}_{\LeRi}^{\bpsi} \right), \quad n \in \Z,
		\end{multline*}
		see \cite[(3.3.4)]{Gin18} where a description of $\mathfrak{n}_{\LeRi, i}^{\bpsi}$ can be found in a more general context.
	\end{itemize}
	
	The entire diagram commutes, completing the proof.
\end{proof}

We record the following ring-theoretic consequences.

\begin{proposition}\label{prop:regular-elements}
	Let $w \in \Wh \smallsetminus \{0\}$, then $w$ is not a zero-divisor on either side (i.e.\ $wx = 0 \iff x = 0 \iff xw = 0$ for all $x \in \Wh$).
\end{proposition}
\begin{proof}
	As the filtration $\mathrm{F}_\bullet \Wh$ is separating by Theorem \ref{prop:separating}, the problem reduces immediately to $\CC[\mathfrak{Z}]$.
\end{proof}

\begin{proposition}\label{prop:Ore-subset}
	For all nonzero $f \in \CC[G]^{N_{\Le} \times N_{\Ri}} \subset \Wh$, the subset $f^{\Z_{\geq 0}}$ is an Ore set in $\Wh$ (see Subsection \ref{subsec:Ore-sets}).
\end{proposition}
\begin{proof}
	The commutator endomorphism $[f, \cdot]$ is locally nilpotent on $\Wh$, since it is already so on $D(G)$. This fact readily implies the required property, as in \cite[Lemma 4.8]{LS99}.
\end{proof}

\subsection{Relation with degenerate nil-DAHA}\label{subsec:DAHA}
The goal here is to fix notation following \cite[Section 7.2]{Gin18}; see also the foundational work \cite{KK86} on nil-Hecke algebras. Take $B = TN \subset G$ and write $\mathrm{W} = \mathrm{W}(G, T)$ as usual, then take
\begin{itemize}
	\item $\mathfrak{g}_{\mathrm{aff}} \supset \mathfrak{t}_{\mathrm{aff}}$: the affinization of $\mathfrak{g} \supset \mathfrak{t}$;
	\item $\Sigma_{\mathrm{aff}}$: the set of simple affine roots;
	\item $\mathrm{W}_{\mathrm{aff}} \supset \mathrm{W}$ the affine Weyl group generated by $\{s_\alpha: \alpha \in \Sigma_{\mathrm{aff}} \}$ acting on $\mathfrak{t}_{\mathrm{aff}}$ and on its linear dual $\mathfrak{t}_{\mathrm{aff}}^*$, where $s_\alpha$ stands for the root reflection attached to an affine root $\alpha$;
	\item $\hbar \in \mathfrak{t}_{\mathrm{aff}}$: the minimal imaginary positive coroot, so that
	\[ \mathfrak{t}_{\mathrm{aff}} = \mathfrak{t} \oplus \CC\hbar, \quad \Sym(\mathfrak{t}_{\mathrm{aff}}) = \Sym(\mathfrak{t})[\hbar]. \]
\end{itemize}
As a matter of fact, $\hbar$ is $\mathrm{W}$-invariant. We refer to \cite[Section 2.1]{BKP16} for details on affine Lie algebras.

The \emph{degenerate nil-DAHA} (double affine Hecke algebra) mentioned in the title is constructed within the algebra $\mathrm{W}_{\mathrm{aff}} \rtimes \CC(\mathfrak{t}_{\mathrm{aff}}^*)$: it is the $\CC$-algebra $\mathscr{H}(\mathfrak{t}_{\mathrm{aff}}, \mathrm{W}_{\mathrm{aff}})$ generated by the subalgebra $\Sym(\mathfrak{t}_{\mathrm{aff}})$ together with the Demazure elements
\begin{equation*}
	\theta_\alpha := \frac{s_\alpha - 1}{\check{\alpha}}, \quad \alpha \in \Sigma_{\mathrm{aff}}.
\end{equation*}
Set $m_{\alpha\beta} := \mathrm{ord}(s_\alpha s_\beta)$ for all $\alpha, \beta \in \Sigma_{\mathrm{aff}}$. By \cite[Proposition 7.1.2]{Gin18}, the relations in $\mathscr{H}(\mathfrak{t}_{\mathrm{aff}}, \mathrm{W}_{\mathrm{aff}})$ are generated by
\begin{gather*}
	\theta_\alpha^2 = 0, \quad (\theta_\alpha \theta_\beta)^{m_{\alpha\beta}} = (\theta_\beta \theta_\alpha)^{m_{\alpha\beta}}, \\
	\theta_\alpha s_\alpha(h) - h \theta_\alpha = \lrangle{\alpha, h},
\end{gather*}
where $h \in \mathfrak{t}_{\mathrm{aff}}$ and $\alpha, \beta \in \Sigma_{\mathrm{aff}}$.

Note that $\mathrm{W}_{\mathrm{aff}} \rtimes \CC(\mathfrak{t}_{\mathrm{aff}}^*)$ acts on $\CC(\mathfrak{t}_{\mathrm{aff}}^*)$. The subspace $\Sym(\mathfrak{t}_{\mathrm{aff}})$ of $\CC(\mathfrak{t}_{\mathrm{aff}}^*)$ affords a standard $\mathscr{H}(\mathfrak{t}_{\mathrm{aff}}, \mathrm{W}_{\mathrm{aff}})$-module, which is faithful.

Let $Q \subset \mathbf{X}^*(T)$ be the root lattice. Let $\widetilde{\mathrm{W}} := \mathrm{W} \ltimes \mathbf{X}^*(T) \supset \mathrm{W}_{\mathrm{aff}}$ be the extended Weyl group. It also acts on $\mathfrak{t}_{\mathrm{aff}}$, and $\widetilde{\mathrm{W}} \simeq (\mathbf{X}^*(T)/Q) \ltimes \mathrm{W}_{\mathrm{aff}}$. We enlarge $\mathscr{H}(\mathfrak{t}_{\mathrm{aff}}, \mathrm{W}_{\mathrm{aff}})$ by taking the smash product
\begin{equation*}
	\mathscr{H}(\mathfrak{t}_{\mathrm{aff}}, \widetilde{\mathrm{W}}) := \mathbf{X}^*(T) \ltimes_Q \mathscr{H}(\mathfrak{t}_{\mathrm{aff}}, \mathrm{W}_{\mathrm{aff}}).
\end{equation*}
Therefore, the group algebra of $\mathbf{X}^*(T)$ embeds into $\mathscr{H}(\mathfrak{t}_{\mathrm{aff}}, \widetilde{\mathrm{W}})$. Note that $\mathscr{H}(\mathfrak{t}_{\mathrm{aff}}, \widetilde{\mathrm{W}})$ still acts naturally and faithfully on $\Sym(\mathfrak{t}_{\mathrm{aff}})$.

For each $\mu \in \mathbf{X}^*(T)$, let $e^\mu$ denote the corresponding element of $\mathscr{H}(\mathfrak{t}_{\mathrm{aff}}, \widetilde{\mathrm{W}})$.

As verified in \cite[Section 7.2]{Gin18}, $\mathscr{H}(\mathfrak{t}_{\mathrm{aff}}, \widetilde{\mathrm{W}})$ is actually a $\Z$-graded $\CC[\hbar]$-algebra, such that
\begin{itemize}
	\item it is $\hbar$-torsion-free, which can also be deduced from \cite[Theorem 4.6]{KK86};
	\item $\Sym(\mathfrak{t}_{\mathrm{aff}})$ gains its natural $\Z_{\geq 0}$-grading;
	\item $\theta_\alpha$ is homogeneous of degree $-1$ for each $\alpha \in \Sigma_{\mathrm{aff}}$.
\end{itemize}

For all $c \in \CC$, follow the conventions of Subsection \ref{subsec:fil-completions} to define
\begin{equation*}
	\mathscr{H}_c := \mathscr{H}(\mathfrak{t}_{\mathrm{aff}}, \widetilde{\mathrm{W}})|_{\hbar = c}, \quad \Hk := \mathscr{H}_1.
\end{equation*}
Therefore $\Hk$ is a filtered $\CC$-algebra.

From \cite[Theorem 7.1.4]{Gin18}, the group algebra of $\mathrm{W}$ embeds naturally in $\mathscr{H}(\mathfrak{t}_{\mathrm{aff}}, \widetilde{\mathrm{W}})$, and the resulting $\mathrm{W}$-action on $\Sym(\mathfrak{t})$ is the standard one. Define the idempotent accordingly
\begin{equation*}
	\mathbf{e} := \frac{1}{|\mathrm{W}|} \sum_{w \in \mathrm{W}} w.
\end{equation*}
The spherical parts are defined as
\begin{equation*}
	\mathscr{H}(\mathfrak{t}_{\mathrm{aff}}, \widetilde{\mathrm{W}})^{\mathrm{sph}} := \mathbf{e} \mathscr{H}(\mathfrak{t}_{\mathrm{aff}}, \widetilde{\mathrm{W}}) \mathbf{e}, \quad \mathscr{H}_c^{\mathrm{sph}} := \mathbf{e} \mathscr{H}_c \mathbf{e}, \quad \Hk^{\mathrm{sph}} := \mathbf{e}\Hk\mathbf{e}.
\end{equation*}
The algebra $\Hk^{\mathrm{sph}}$ is still filtered and contains $\Sym(\mathfrak{t})^{\mathrm{W}}$ as a subalgebra.

Recall that we put the PBW filtration on $\mathcal{Z}(\mathfrak{g})$. The corresponding Rees algebra is denoted as $\mathcal{Z}(\mathfrak{g})_\hbar$.

\begin{theorem}[\protect{\cite[Theorems 1.2.1 and 8.1.2]{Gin18}}]
	\label{prop:nil-DAHA}
	Suppose $G$ is simply connected. There is a commutative diagram of graded algebras:
	\[\begin{tikzcd}
		\left( \Sym(\mathfrak{t})^{\mathrm{W}} \right)[\hbar] \arrow[r, "\sim"] \arrow[d] & (\mathcal{Z}\mathfrak{g})_{\hbar} \arrow[d] \\
		\mathscr{H}(\mathfrak{t}_{\mathrm{aff}}, \widetilde{\mathrm{W}})^{\mathrm{sph}} \arrow[r, "\sim"'] & \Wh_{\hbar}.
	\end{tikzcd}\]
	Specializing at $\hbar = 1$, we obtain an isomorphism $\Wh \simeq \Hk^{\mathrm{sph}}$ between filtered $\CC$-algebras, which restricts to $\mathcal{Z}(\mathfrak{g}) \simeq \Sym(\mathfrak{t})^{\mathrm{W}}$, and the latter isomorphism coincides with Harish-Chandra's isomorphism.
\end{theorem}

\subsection{Flatness}
In this subsection, we exploit the connection with degenerate nil-DAHA to prove the following results.

\begin{proposition}\label{prop:Z-flatness}
	The algebra $\Wh$ is flat as a left (resp.\ right) $\mathcal{Z}(\mathfrak{g})$-module.
\end{proposition}
\begin{proof}
	To begin with, assume $G$ is simply connected. By Theorem \ref{prop:nil-DAHA}, it suffices to show that $\Hk^{\mathrm{sph}}$ is flat as a left (resp.\ right) $\Sym(\mathfrak{t})^{\mathrm{W}}$-module.
	
	First, \cite[Theorem 4.6]{KK86} implies that $\mathscr{H}(\mathfrak{t}_{\mathrm{aff}}, \mathrm{W}_{\mathrm{aff}})$ is free as a left (resp.\ right) $\Sym(\mathfrak{t}_{\mathrm{aff}})$-module. Taking smash products with $\mathbf{X}^*(T)/Q$, the same is true for $\mathscr{H}(\mathfrak{t}_{\mathrm{aff}}, \widetilde{\mathrm{W}})$.
	
	We focus on the version for right modules. Specialization to $\hbar = 1$ shows $\Hk$ is free over $\Sym(\mathfrak{t})$, thus free over $\Sym(\mathfrak{t})^{\mathrm{W}}$. Since $\Sym(\mathfrak{t})^{\mathrm{W}} \subset \Hk$ commutes with $\mathbf{e}$ and $\mathbf{e}$ is idempotent, $\Hk\mathbf{e}$ is a direct summand of $\Hk$ as $\Sym(\mathfrak{t})^{\mathrm{W}}$-modules, and so is $\Hk^{\mathrm{sph}} = \mathbf{e}\Hk\mathbf{e}$. This proves the flatness of $\Hk^{\mathrm{sph}}$.
	
	Next, consider the case where $G = T$ is a torus. Then $\Wh = D(T)$, $\mathcal{Z}(\mathfrak{g}) = \Sym(\mathfrak{t})$ and the result is obvious. For the general case, take the isogeny $G' \twoheadrightarrow G$ with finite central kernel $Z$ as in \eqref{eqn:G-SES}. Consider the decomposition into $Z$-eigenspaces
	\[ \Wh_{G'} = \bigoplus_{\chi: Z \to \CC^{\times}} \Wh_{G'}^\chi. \]
	Each summand above is a $\mathcal{Z}(\mathfrak{g}')$-submodule. From \eqref{eqn:W-isogeny} we get
	\[ \Wh_G = \Wh_{G'}^{Z\text{-inv}}, \quad \mathcal{Z}(\mathfrak{g}) = \mathcal{Z}(\mathfrak{g}'). \]
	
	The flatness of $\Wh_{G'}$ over $\mathcal{Z}(\mathfrak{g})$ is known. Being the direct summand with $\chi = \mathrm{triv}$ of $\Wh_{G'}$, it follows that $\Wh = \Wh_G$ is flat as well.
\end{proof}

\begin{proposition}\label{prop:Z-summand}
	View the Rees algebra $\Wh_{\hbar}$ as a left (resp.\ right) $(\mathcal{Z}\mathfrak{g})_{\hbar}$-module, compatibly with gradings. Then there exists a decomposition
	\[ \Wh_{\hbar} = (\mathcal{Z}\mathfrak{g})_{\hbar} \oplus L \]
	into a direct sum of $(\mathcal{Z}\mathfrak{g})_{\hbar}$-stable graded subspaces, and $L$ can be chosen to be $Z_G$-invariant.
	
	Consequently, $\Wh$ is faithfully flat as a left (resp.\ right) $\mathcal{Z}(\mathfrak{g})$-module.
\end{proposition}
\begin{proof}
	The second assertion follows from the first by specialization at $\hbar = 1$.
	
	To prove the first assertion about direct sum, first suppose that $G$ is simply connected, then a precise statement for left modules is in \cite[Lemma 8.2.1]{Gin18}, and the arguments there apply to right modules as well. The explicit definition of $L$ in \textit{loc.\ cit.} shows that it is $Z_G$-invariant.
	
	Alternatively, for the simply connected case one can translate the problem to the nil-DAHA side via Theorem \ref{prop:nil-DAHA}, and deduce it from \cite[Theorem 4.6]{KK86}.

	Next, the case $G = T$ is clear. For the general case, take the isogeny \eqref{eqn:G-SES} and argue as in Proposition \ref{prop:Z-flatness} via the identification $(\Wh_{G'})_\hbar^{Z\text{-inv}} = (\Wh_G)_{\hbar}$, by noting that $(\mathcal{Z}\mathfrak{g})_\hbar = (\mathcal{Z}\mathfrak{g}')_\hbar$ and $Z$ stabilizes $\hbar$.
\end{proof}

\section{On certain \texorpdfstring{$\Wh$}{W}-modules}\label{sec:W-modules}
Retain the conventions made in Section \ref{sec:W-properties}. Unless otherwise stated, the modules below are left modules.

\subsection{Review of monodromic modules}
Denote by $(D(G), \mathfrak{n}_{\LeRi}^{\bpsi})\dcate{Mod}$ the full subcategory of $D(G)\dcate{Mod}$ consisting of modules on which $\mathfrak{n}_{\LeRi}^{\bpsi} \subset \mathcal{U}\mathfrak{g}_{\Le} \otimes \mathcal{U}\mathfrak{g}_{\Ri}$ acts locally nilpotently. This is a Serre subcategory, and should be regarded as the category of $(N_{\Le} \times N_{\Ri}, \bpsi \times \bpsi)$-monodromic $D(G)$-modules.

For every object $M^\natural$ of $(D(G), \mathfrak{n}_{\LeRi}^{\bpsi})\dcate{Mod}$, denote by $(M^\natural)^{\mathfrak{n}_{\LeRi}^{\bpsi}} \subset M^\natural$ the subspace annihilated by $\mathfrak{n}_{\LeRi}^{\bpsi}$; it is then a $\Wh$-module.

Note that $D(G)/D(G)\mathfrak{n}_{\LeRi}^{\bpsi}$ is a $(D(G), \Wh)$-bimodule. By \cite[Proposition 3.1.4 (i)]{Gin18}, there is an equivalence of categories à la Skryabin:
\begin{equation}\label{eqn:monodromic-equivalence}
	\begin{tikzcd}[row sep=tiny]
		(D(G), \mathfrak{n}_{\LeRi}^{\bpsi})\dcate{Mod} \arrow[r, shift left] & \Wh\dcate{Mod} \arrow[l, shift left] \\
		M^{\natural} \arrow[mapsto, r] & (M^\natural)^{\mathfrak{n}_{\LeRi}^{\bpsi}} \\
		\left( D(G)/D(G)\mathfrak{n}_{\LeRi}^{\bpsi} \right) \dotimes{\Wh} M & M \arrow[mapsto, l];
	\end{tikzcd}
\end{equation}
the unit and co-unit morphisms here are the evident ones.

For brevity, the equivalence \eqref{eqn:monodromic-equivalence} will be written as $M^\natural \leftrightarrow M$. The simple fact below will frequently be invoked without explicit mention.

\begin{lemma}
	The $\Wh$-module $M$ is finitely generated if and only if $M^\natural$ is finitely generated as a $D(G)$-module.
\end{lemma}
\begin{proof}
	The ``only if'' direction is obvious from \eqref{eqn:monodromic-equivalence}. Conversely, if $M^\natural$ is finitely generated over $D(G)$, then any family of proper subobjects $(M_\alpha^\natural)_{\alpha \in \mathcal{A}}$ of $M^\natural$, where $\mathcal{A}$ is a filtered poset, has a proper supremum in $(D(G), \mathfrak{n}_{\LeRi}^{\bpsi})\dcate{Mod}$ given by $\bigcup_\alpha M_\alpha^\natural$; this categorical property transmits to $M$ and implies $M$ is finitely generated over $\Wh$.
\end{proof}

\subsection{Admissible modules}
In what follows, the embedding $\mathcal{Z}(\mathfrak{g}) \hookrightarrow \Wh$ will often be omitted in the notation.

\begin{definition}\label{def:admissible-module}
	A $\Wh$-module is said to be \emph{admissible} if it is finitely generated over $\Wh$, and locally $\mathcal{Z}(\mathfrak{g})$-finite. Denote the full subcategory of $\Wh\dcate{Mod}$ formed by admissible $\Wh$-modules as $\cate{Adm}$.
\end{definition}

It is clear that $\cate{Adm}$ is a Serre subcategory of $\Wh\dcate{Mod}$.

\begin{lemma}\label{prop:admissible-generation}
	If a $\Wh$-module $M$ is generated by finitely many $\mathcal{Z}(\mathfrak{g})$-finite elements, then $M$ is admissible.
\end{lemma}
\begin{proof}
	Suppose $M$ is attached to $M^\natural$ as in \eqref{eqn:monodromic-equivalence}. Let $x_1, \ldots, x_n$ be $\mathcal{Z}(\mathfrak{g})$-finite generators of $M$, and let $M_0$ be the $\mathfrak{g}$-submodule of $M^\natural$ they generate. Consider the linear surjection
	\[ D(G) \otimes M_0 \twoheadrightarrow M^\natural , \quad P \otimes x \mapsto Px. \]
	With the $\mathfrak{g}$-module structure on $D(G)$ (resp.\ $D(G) \otimes M_0$) given by $\xi \cdot P := [\xi, P]$ where the commutator is taken in $D(G)$ (resp.\ the tensor product $\mathfrak{g}$-module structure), the map above is $\mathfrak{g}$-linear. It suffices to show that $D(G) \otimes M_0$ is locally $\mathcal{Z}(\mathfrak{g})$-finite.
	
	The $\mathfrak{g}$-module $D(G)$ defined above lifts naturally to a $G$-module, thus $D(G)$ is a filtered union of finite-dimensional submodules. To conclude the proof, it remains to apply a result of Kostant \cite[Theorem 7.133]{KV95} to $D(G) \otimes M_0$.
\end{proof}

The following discussions are based on the isomorphisms furnished by Theorem \ref{prop:grW-Z}.

\begin{definition}
	Recall that the filtration on $\mathcal{Z}(\mathfrak{g})$ induced from $\mathrm{F}_\bullet \Wh$ is the PBW filtration (Proposition \ref{prop:PBW-center}). Denote by $J^+$ the augmentation ideal of $\gr(\mathcal{Z}(\mathfrak{g}))$; since $\gr_0(\mathcal{Z}(\mathfrak{g})) = \CC = \gr_0 (\mathcal{U}(\mathfrak{g}))$, we see that $J^+$ contains every proper graded ideal of $\gr(\mathcal{Z}(\mathfrak{g}))$ and
	\[ \gr(\mathcal{Z}(\mathfrak{g})) \big/ J^+ \simeq \CC. \]
	In particular, $J^+$ is a maximal ideal. Let $\nu_0 \in \mathfrak{c}$ be the point corresponding to $J^+$.
\end{definition}

\begin{lemma}\label{prop:SS-bound}
	Let $M$ be an admissible $\Wh$-module. Then
	\[ \mathrm{SS}(M) \subset \kappa^{-1}(\nu_0), \]
	and equality holds when $\mathrm{SS}(M)$ is nonempty.
\end{lemma}
\begin{proof}
	Take $\mathcal{Z}(\mathfrak{g})$-finite generators $x_1, \ldots, x_n \in M$. There exists an ideal $I \subset \mathcal{Z}(\mathfrak{g})$ of finite codimension annihilating all $x_i$. Equip $I$ and $\mathcal{Z}(\mathfrak{g})/I$ with the filtrations induced from $\mathcal{Z}(\mathfrak{g})$. Then $\gr\left(\mathcal{Z}(\mathfrak{g})\right)/\gr(I) \simeq \gr\left(\mathcal{Z}(\mathfrak{g})/I\right)$ is finite-dimensional.
	
	If $\gr(I)$ is proper, then $\gr(I) \subset J^+$; on the other hand, by reasons of codimension $\gr_k(I) = \gr_k \mathcal{Z}(\mathfrak{g})$ for $k \gg 0$, hence $\gr(I) \supset (J^+)^k$ for $k \gg 0$.
	
	Equip $M$ with the good filtration \eqref{eqn:good-filtration-fg} determined by $x_1, \ldots, x_n$. For every $1 \leq i \leq n$, the image of $x_i$ in $\gr_0(M)$ is annihilated by $\gr(I)$. These images generate $\gr(M)$ over $\gr(\Wh)$, hence $\mathrm{SS}(M) \subset V(\gr(I) \gr(\Wh))$, but $V(\gr(I) \gr(\Wh))$ is either empty (when $\gr(I) = \gr(\mathcal{Z}(\mathfrak{g}))$) or equal to $V(J^+ \gr(\Wh)) = \kappa^{-1}(\nu_0)$.
	
	This establishes the inclusion. When $G$ is adjoint, the fiber $\kappa^{-1}(\nu_0)$ is connected (see the proof of Lemma \ref{prop:fiber-cell}), hence irreducible since it is an algebraic group. When $G$ is the product of an adjoint group with a torus, then $\kappa^{-1}(\nu_0)$ is also irreducible. In both cases, since the fibers of $\kappa$ are Lagrangian in $\mathfrak{Z}$, by Theorem \ref{prop:integrability} we must have $\mathrm{SS}(M) = \kappa^{-1}(\nu_0)$.
	
	For the general case, we leverage the isogeny $G \to G_{\mathrm{ab}} \times G/Z_G$ as in the proof of Lemma \ref{prop:fiber-cell} to conclude.
\end{proof}

\begin{remark}
	The precise value $\nu_0$ is not needed in the sequel. However, by identifying $\mathfrak{c}$ with $\CC^{\mathrm{rk}(G)}$ via fundamental invariants, the discussions in \cite[Section 6]{BG17} entail that $\kappa$ is $\Gm$-equivariant if we let $\Gm$ act on $\mathfrak{c}$ by twice the exponents, which in turn imply $\nu_0 = 0$.
\end{remark}

Next, we inspect the $(D(G), \mathfrak{n}_{\LeRi}^{\bpsi})$-module corresponding to an admissible $\Wh$-module.

\begin{proposition}\label{prop:connection-big-cell}
	Let $M$ be an admissible $\Wh$-module and $M^\natural$ the corresponding object of $(D(G), \mathfrak{n}_{\LeRi}^{\bpsi})\dcate{Mod}$ via \eqref{eqn:monodromic-equivalence}. Then $M^\natural$ is holonomic as a $D(G)$-module, and its restriction to $Bw_0 B$ is a flat connection.
\end{proposition}
\begin{proof}
	Due to the local nilpotence of the $\mathfrak{n}_{\LeRi}^{\bpsi}$-action, $M^\natural$ is $\mathfrak{n}_{\LeRi}$-admissible in the sense of \cite[Definition 3.1]{Li22}. As $N \times N$ is a spherical subgroup of $G \times G$, by \cite[Corollary 3.9]{Li22} we know $M^\natural$ is holonomic. Alternatively, one can invoke \cite{AGM16}.
	
	Denoting by $\mathcal{N} \subset \mathfrak{g}_{\Le}^* \times \mathfrak{g}_{\Ri}^*$ the nilpotent cone, \cite[Proposition 3.4]{Li22} gives
	\begin{equation*}
		\mathrm{SS}(M^\natural) \subset (\bm{\mu}_{\Le} \times \bm{\mu}_{\Ri})^{-1}\left( \mathcal{N} \cap (\mathfrak{n}_{\Le} \times \mathfrak{n}_{\Ri})^\perp \right).
	\end{equation*}
	Identifying $\mathfrak{g}$ with $\mathfrak{g}^*$ via $(\cdot, \cdot)$ so that $\mathfrak{n}^\perp = \mathfrak{b}$, the right hand side becomes
	\begin{equation*}
		\left\{ (x', g, x)  \in \mathfrak{g} \times G \times \mathfrak{g} : x, x' \in \mathfrak{n}, \; x' = \Ad(g)(x) \right\}.
	\end{equation*}
	Therefore $x \in \mathfrak{n} \cap g^{-1} \mathfrak{n} g$. If $g \in B w_0 B$ then $(\mathfrak{b}, g^{-1}\mathfrak{b}g) = (h\mathfrak{b}h^{-1}, h\mathfrak{b}^- h^{-1})$ for some $h \in G$, hence $\mathfrak{n} \cap g^{-1} \mathfrak{n} g = \{0\}$. This shows $M^\natural$ is a flat connection over $B w_0 B$.
\end{proof}

\begin{remark}
	The flat connection $j^* M^\natural$ on $Bw_0 B$ is never regular, except when $j^* M^\natural$ is zero or when $G$ is a torus. This is due to its exponential behavior along $N_{\Le} \times N_{\Ri}$-orbits.
\end{remark}

\begin{corollary}\label{prop:finite-length}
	Admissible $\Wh$-modules are of finite length.
\end{corollary}
\begin{proof}
	Use the equivalence \eqref{eqn:monodromic-equivalence}, Proposition \ref{prop:connection-big-cell} and the fact that holonomic $D(G)$-modules are of finite length.
\end{proof}

\subsection{Harish-Chandra modules}
For each $\nu \in \mathfrak{c}$, let $\mathfrak{m}_\nu$ be the corresponding maximal ideal of $\mathcal{Z}(\mathfrak{g})$.

\begin{definition}\label{def:HC-module}
	The Harish-Chandra $\Wh$-module with infinitesimal character $\nu$ is the $\Wh$-module
	\[ \mathcal{G}_\nu := \Wh / \Wh\mathfrak{m}_\nu. \]
\end{definition}

\begin{remark}\label{rem:HC-admissible}
	Harish-Chandra $\Wh$-modules are clearly admissible. Conversely, given a nonzero admissible $\Wh$-module $M$, one may take $x \in M \smallsetminus \{0\}$ annihilated by some $\mathfrak{m}_{\nu_1} \subset \mathcal{Z}(\mathfrak{g})$, so that $M_1 := \Wh x \subset M$ is a quotient of $\mathcal{G}_{\nu_1}$. By induction on length, we see that $M$ admits a filtration
	\[ \{0\} = M_0 \subset \cdots \subset M_n = M \]
	such that each $M_i/M_{i-1}$ is a quotient of $\mathcal{G}_{\nu_i}$ for some $\nu_i \in \mathfrak{c}$.
\end{remark}

\begin{remark}
	Via the equivalence in \cite[Proposition 3.1.4 (i)]{Gin18}, $\mathcal{G}_\nu$ corresponds to the $(N_{\Le} \times N_{\Ri}, \bpsi \times \bpsi)$-monodromic $D(G)$-module
	\[ \mathcal{G}_\nu^\natural := D(G) \big/ \left( D(G) \mathfrak{n}_{\LeRi}^{\bpsi} + D(G)\mathfrak{m}_\nu \right). \]
\end{remark}

\begin{remark}
	If $G = G'/Z$ where $G'$ is connected reductive and $Z \subset G'$ is a finite central subgroup, then the corresponding Harish-Chandra $\Wh_{G'}$-module $\mathcal{G}'_\nu$ carries a natural $Z$-action with $(\mathcal{G}'_\nu)^{Z\text{-inv}} \simeq \mathcal{G}_\nu$, compatible with $\Wh_G = \Wh_{G'}^{Z\text{-inv}}$; see \eqref{eqn:W-isogeny}.
\end{remark}

The completion of $\mathcal{G}_\nu$ can be identified as
\[ \widehat{\mathcal{G}}_\nu \simeq \widehat{\Wh} \dotimes{\Wh} (\Wh/\Wh\mathfrak{m}_\nu) \simeq \widehat{\Wh}/\widehat{\Wh}\mathfrak{m}_\nu, \]
and $\mathrm{SS}(\mathcal{G}_\nu) = \mathrm{SS}(\widehat{\mathcal{G}}_\nu)$ (Proposition \ref{prop:SS-completion}).

\begin{proposition}\label{prop:Gnu-nonzero}
	We have $\widehat{\mathcal{G}}_\nu \neq 0$ and $\mathrm{SS}(\mathcal{G}_\nu) = \mathrm{SS}(\widehat{\mathcal{G}}_\nu) = \kappa^{-1}(\nu_0)$.
\end{proposition}
\begin{proof}
	It suffices to prove that $\widehat{\mathcal{G}}_\nu \neq 0$ for any good filtration, which will imply $\gr(\widehat{\mathcal{G}}_\nu) \neq 0$ and then $\mathrm{SS}(\mathcal{G}_\nu) = \mathrm{SS}(\widehat{\mathcal{G}}_\nu) = \kappa^{-1}(\nu_0)$, by Lemma \ref{prop:SS-bound}.
	
	Put the quotient filtration induced from $\mathrm{F}_\bullet \Wh$ on $\mathcal{G}_\nu$. Suppose on the contrary that $\widehat{\mathcal{G}}_\nu = 0$, then Lemma \ref{prop:null-completion} would imply
	\begin{equation}\label{eqn:Gnu-nonzero}
		\forall n \in \Z, \; \mathrm{F}_n \Wh + \Wh\mathfrak{m}_\nu = \Wh.
	\end{equation}
	
	The graded direct sum decomposition in Proposition \ref{prop:Z-summand} specialized at $\hbar = 1$ yields
	\[ \Wh = \mathcal{Z}(\mathfrak{g}) \oplus (L|_{\hbar = 1}), \]
	a direct sum decomposition of \emph{filtered} right $\mathcal{Z}(\mathfrak{g})$-modules, where $\mathcal{Z}(\mathfrak{g})$ carries the PBW filtration $\mathrm{F}^{\mathrm{PBW}}_\bullet \mathcal{Z}(\mathfrak{g})$. Taking the projection of \eqref{eqn:Gnu-nonzero} to the factor $\mathcal{Z}(\mathfrak{g})$ would give
	\[ \forall n \in \Z, \; \mathrm{F}^{\mathrm{PBW}}_n \mathcal{Z}(\mathfrak{g}) + \mathcal{Z}(\mathfrak{g}) \mathfrak{m}_\nu = \mathcal{Z}(\mathfrak{g}). \]
	This is clearly absurd when $n < 0$.
\end{proof}

All definitions and results above have counterparts for right $\Wh$-modules. Viewing $\Wh$ as a $(\Wh, \Wh)$-bimodule, $\Ext^m_{\Wh}(M, \Wh)$ becomes a right $\Wh$-module for every left $\Wh$-module $M$ and $m \in \Z_{\geq 0}$. Similarly for $\Ext^m_{\widehat{\Wh}}(\widehat{M}, \widehat{\Wh})$ for every left $\widehat{\Wh}$-module $\widehat{M}$.

The following variant of \cite[Corollary 4.7]{BNS24} will be used later on.

\begin{lemma}\label{prop:Ext-HC}
	Denote by $\mathcal{G}_\nu^{\mathrm{r}}$ the right version of $\mathcal{G}_\nu$. Let $n = \dim \mathfrak{c} = \dim T$. For all $\nu \in \mathfrak{c}$, we have
	\begin{gather*}
		\Ext^m_{\Wh}(\mathcal{G}_\nu, \Wh) \simeq \begin{cases}
			0, & \text{if}\; m > n \\
			\mathcal{G}_\nu^{\mathrm{r}}, & \text{if}\; m = n,
		\end{cases} \\
		\Ext^m_{\widehat{\Wh}}(\widehat{\mathcal{G}}_\nu, \widehat{\Wh}) \simeq \begin{cases}
			0, & \text{if}\; m > n \\
			\widehat{\mathcal{G}}_\nu^{\mathrm{r}}, & \text{if}\; m = n,
		\end{cases}
	\end{gather*}
	as right $\Wh$-modules and $\widehat{\Wh}$-modules, respectively.
\end{lemma}
\begin{proof}
	Pick a $\CC$-subspace $U \subset \mathfrak{m}_\nu$ with $\dim U = n$ such that $U$ generates the ideal $\mathfrak{m}_\nu$ of $\mathcal{Z}(\mathfrak{g})$. The Koszul complex yields a free resolution
	\[ \cdots \to 0 \to \mathcal{Z}(\mathfrak{g}) \otimes \bigwedge\nolimits^n U \to \cdots \to \mathcal{Z}(\mathfrak{g}) \otimes \bigwedge\nolimits^0 U \to \mathcal{Z}(\mathfrak{g})/\mathfrak{m}_\nu \to 0 \]
	of the $\mathcal{Z}(\mathfrak{g})$-module $\mathcal{Z}(\mathfrak{g})/\mathfrak{m}_\nu$. In view of Proposition \ref{prop:Z-flatness}, tensoring with $\Wh$ over $\mathcal{Z}(\mathfrak{g})$ yields a projective resolution
	\[ \cdots \to 0 \to \Wh \otimes \bigwedge\nolimits^n U \to \cdots \to \Wh \otimes \bigwedge\nolimits^0 U \to \mathcal{G}_\nu \to 0 \]
	of $\mathcal{G}_\nu$. Now apply $\Hom_{\Wh}(\cdot, \Wh)$ to obtain $\Ext_{\Wh}^m(\mathcal{G}_\nu, \Wh) = 0$ when $m > n$. As for $m = n$, we fix a basis of $U$ and obtain
	\begin{align*}
		\Ext^n_{\Wh}(\mathcal{G}_\nu, \Wh) & \simeq \Coker\left[ \Hom_{\Wh}(\Wh \otimes \bigwedge\nolimits^{n-1} U, \Wh) \to \Hom_{\Wh}(\Wh \otimes \bigwedge\nolimits^n U, \Wh) \simeq \Wh \right] \\
		& \simeq \Wh/\mathfrak{m}_\nu \Wh = \mathcal{G}_\nu^{\mathrm{r}}
	\end{align*}
	by the description of Koszul complexes.
	
	The case for $\widehat{\Wh}$-modules is similar: simply take tensor products with $\widehat{\Wh}$ instead of $\Wh$ in the preceding arguments, and note that the composite of $\mathcal{Z}(\mathfrak{g}) \to \Wh \to \widehat{\Wh}$ is still flat by Proposition \ref{prop:completion-tensor} (iii).
\end{proof}

There is a version with left and right swapped, which we omit.

By Proposition \ref{prop:regular-elements}, $\Ext^0_{\Wh}(\mathcal{G}_\nu, \Wh) = \Hom_{\Wh}(\mathcal{G}_\nu, \Wh) = 0$. As for $\Ext^i_{\Wh}(\mathcal{G}_\nu, \Wh)$ where $0 < i < n$, Lemma \ref{prop:Ext-HC} provides no information; in Subsection \ref{subsec:duality} they will be shown to vanish after completion.

\subsection{Torsion-freeness}
Choose an element $d \in \CC[G]^{N_{\Le} \times N_{\Ri}} \smallsetminus \{0\}$ as in Lemma \ref{prop:d}, such that the defining ideal $\mathcal{I}^\natural$ of $G \smallsetminus Bw_0 B$ is $\sqrt{(d)}$. Define
\[ \mathcal{S} := \left\{ d^k : k \in \Z_{\geq 0} \right\}, \]
a multiplicative subset of $\Wh$, and embed it into $\widehat{\Wh}$.

\begin{definition}\label{def:S-torsion}
	An element of a $\Wh$-module (resp.\ $\widehat{\Wh}$-module) is said to be of $\mathcal{S}$-torsion if its annihilator intersects $\mathcal{S}$.

	A $\Wh$-module $M$ (resp.\ $\widehat{\Wh}$-module) is said to be of $\mathcal{S}$-torsion if it can be generated by finitely many elements of $\mathcal{S}$-torsion.
\end{definition}

For $\Wh$-modules, an alternative description is available.

\begin{lemma}
	A $\Wh$-module $M$ is of $\mathcal{S}$-torsion if and only if all elements of $M$ are of $\mathcal{S}$-torsion.
\end{lemma}
\begin{proof}
	This follows directly from the left Ore condition of $\mathcal{S}$ established in Lemma \ref{prop:Ore-subset}.
\end{proof}

\begin{lemma}\label{prop:support-torsion}
	Given $M^{\natural} \leftrightarrow M$ as in \eqref{eqn:monodromic-equivalence}, and assume $M$ is finitely generated. The following are equivalent:
	\begin{enumerate}[(i)]
		\item $M^\natural$ is supported off $Bw_0 B$ as a $\mathscr{O}_G$-module;
		\item every element of $M^\natural$ is annihilated by some element of $\mathcal{S}$;
		\item $M$ is of $\mathcal{S}$-torsion.
	\end{enumerate}
\end{lemma}
\begin{proof}
	(i) $\iff$ (ii). The support condition amounts to $M^\natural [d^{-1}] = 0$ as $\mathscr{O}_G$-module, i.e.\ every element $M^\natural$ is annihilated by some power of $d$.
	
	(ii) $\implies$ (iii). The natural $\CC[G]^{N_{\Le} \times N_{\Ri}}$-module homomorphism $M \to M^\natural$ is injective, which follows from the description of the equivalence \eqref{eqn:monodromic-equivalence}.
	
	(iii) $\implies$ (ii). Elements of $M^\natural$ can be expressed as linear combinations of $\overline{P} \otimes x$ where $P \in D(G)$ with image $\overline{P}$ in $D(G)/D(G)\mathfrak{n}_{\LeRi}^{\bpsi}$, and $x \in M$ is annihilated by some power of $d$. Since $[d, \cdot]$ is a locally nilpotent endomorphism on $D(G)$ (cf.\ the proof of Proposition \ref{prop:Ore-subset}), we see that a sufficiently high power of $d$ annihilates $\overline{P} \otimes x$, for each $\overline{P} \otimes x$ in the linear combination.
\end{proof}

\begin{remark}\label{rem:minimal-extension}
	The relevance of $\mathcal{S}$-torsion can be understood as follows. Let $j: Bw_0 B \hookrightarrow G$ be the open embedding of the big Bruhat cell, and denote by $j^*$ the pull-back functor of $D$-modules. Let $M$ be an admissible $\Wh$-module. If $M$ has neither quotients nor submodules that are non-zero of $\mathcal{S}$-torsion, then \eqref{eqn:monodromic-equivalence} and Lemma \ref{prop:support-torsion} would imply that $M^\natural$ has neither quotients nor submodules that are nonzero and supported off $B w_0 B$, hence
	\[ M^\natural \simeq j_{!*}\left( j^* M^\natural \right) \]
	where the right hand side is the minimal extension of the flat connection $j^* M^\natural$ (Lemma \ref{prop:support-torsion}), by the characterization of minimal extensions.
	
	A result in this direction will be given in Theorem \ref{prop:minimal-extension}.
\end{remark}

Although we are currently unable to prove the absence of nonzero subquotients of $M$ of $\mathcal{S}$-torsion in general, a version up to completion is within reach.

In what follows, completions are always taken relative to good filtrations on finitely generated $\Wh$-modules. Observe that by Definition \ref{def:S-torsion}, the completion of a finitely generated $\Wh$-module of $\mathcal{S}$-torsion is a $\widehat{\Wh}$-module of $\mathcal{S}$-torsion.

\begin{theorem}\label{prop:torsion}
	Let $M$ be an admissible $\Wh$-module. If $Q$ is a subquotient of $\widehat{M}$ such that $Q$ is of $\mathcal{S}$-torsion, then $Q = 0$.
\end{theorem}
\begin{proof}
	By the completeness of $\widehat{\Wh}$ and Remark \ref{rem:Zariskian}, it suffices to show $\mathrm{SS}(Q) = \emptyset$.

	Gabber's Theorem \ref{prop:integrability} for $\widehat{\Wh}$ implies that every irreducible component $C$ of $\mathrm{SS}(Q)$ satisfies $\dim C \geq \frac{1}{2} \dim\mathfrak{Z}$. Therefore, it suffices to show $\dim C < \frac{1}{2} \dim\mathfrak{Z}$ for each $C$.
	
	There exist generators $x_1, \ldots, x_n$ of $Q$ and $k \geq 1$ such that $d^k x_i = 0$ for all $i$. Equip $Q$ with the good filtration \eqref{eqn:good-filtration-fg} determined by $x_1, \ldots, x_n$. For every $i$, the image of $x_i$ in $\gr_0(Q)$ is annihilated by the homogeneous element $d^k$ of $\gr \CC[G]^{N_{\Le} \times N_{\Ri}}$ placed at the adequate degree (see Lemmas \ref{prop:bi-invariant-grading} and \ref{prop:bi-invariant-filtration}). These images generate $\gr(Q)$, hence Lemma \ref{prop:B-d} gives
	\[ \mathrm{SS}(Q) \subset \mathfrak{Z} \smallsetminus \mathfrak{B}_{w_0}. \]
	
	On the other hand, $\mathrm{SS}(Q) \subset \mathrm{SS}(\widehat{M})$ as $Q$ is a subquotient of $\widehat{M}$, and $\mathrm{SS}(\widehat{M}) = \mathrm{SS}(M) \subset \kappa^{-1}(\nu_0)$ by Lemma \ref{prop:SS-bound}. 
	
	It remains to show every irreducible component of $\kappa^{-1}(\nu_0) \cap (\mathfrak{Z} \smallsetminus \mathfrak{B}_{w_0})$ has dimension $< \frac{1}{2} \dim\mathfrak{Z}$. Since $\kappa^{-1}(\nu_0)$ is an algebraic group, its connected and irreducible components are the same, and they all have dimension $\frac{1}{2} \dim\mathfrak{Z}$ since $\kappa^{-1}(\nu_0)$ is Lagrangian. Lemma \ref{prop:fiber-cell} implies that $\mathfrak{B}_{w_0}$ intersects every component of $\kappa^{-1}(\nu_0)$. Therefore the closed subset $\mathfrak{Z} \smallsetminus \mathfrak{B}_{w_0}$ intersects every component of $\kappa^{-1}(\nu_0)$ in a proper subset. The desired inequality follows at once.
\end{proof}

\begin{corollary}\label{prop:torsion-cor}
	If $R$ is a subquotient of $\mathcal{S}$-torsion of an admissible $\Wh$-module $M$, then $\widehat{\Wh} \dotimes{\Wh} R \simeq \widehat{R} = 0$, or equivalently $\gr(R) = 0$ relative to any good filtration.
\end{corollary}
\begin{proof}
	The isomorphism $\widehat{\Wh} \dotimes{\Wh} R \simeq \widehat{R}$ and the equivalence $\widehat{R} = 0 \iff \gr(R) = 0$ follow from Proposition \ref{prop:completion-tensor} and Lemma \ref{prop:null-completion}, respectively. The rest comes from the exactness of completion and Theorem \ref{prop:torsion}.
\end{proof}

Let $M^\natural \leftrightarrow M$ be as in \eqref{eqn:monodromic-equivalence}. If one has a version of Corollary \ref{prop:torsion-cor} for $R$ instead of $\widehat{R}$, then by Lemma \ref{prop:support-torsion} the conclusion would be equivalent to $M^{\natural}$ having no nonzero quotients supported off $Bw_0 B$.

\subsection{Block decomposition}\label{subsec:inf-character}
Identify $\mathfrak{c}$ with $\mathfrak{t}^* / \mathrm{W}$. Suppose a $\Wh$-module $M$ is generated by generalized $\mathcal{Z}(\mathfrak{g})$-eigenvectors $x_1, \ldots, x_n$ with eigenvalues $\nu_1, \ldots, \nu_n \in \mathfrak{c}$, respectively, and $\dot{\nu}_i \in \mathfrak{t}^*$ is a representative of $\nu_i$ for each $i$. Kostant's result \cite[Theorem 7.133]{KV95} cited in the proof of Lemma \ref{prop:admissible-generation} implies that every $x \in M$ is a sum $x = y_1 + \cdots + y_n$ where $y_i$ is a generalized $\mathcal{Z}(\mathfrak{g})$-eigenvector of eigenvalue represented by some element of $\dot{\nu}_i + \mathbf{X}^*(T)$, for all $1 \leq i \leq n$.

Let the extended Weyl group $\widetilde{\mathrm{W}} = \mathrm{W} \ltimes \mathbf{X}^*(T)$ act on $\mathfrak{t}^*$ by affine transformations, where $\mathrm{W} = \mathrm{W}(G, T)$. The set-theoretic quotient map $\mathfrak{t}^* \to \mathfrak{t}^* / \widetilde{\mathrm{W}}$ factors through $\mathfrak{c}$. Elements of $\mathfrak{t}^* / \widetilde{\mathrm{W}}$ will be expressed as $[\nu]$, the image of $\nu \in \mathfrak{c}$, and $\dot{\nu}$ will stand for any representative of $[\nu]$ in $\mathfrak{t}^*$.

\begin{definition}\label{def:block}
	An admissible $\Wh$-module is said to be of \emph{infinitesimal character} $[\nu]$ if it is generated by generalized $\mathcal{Z}(\mathfrak{g})$-eigenvectors $x_1, \ldots, x_n$ whose eigenvalues $\nu_1, \ldots, \nu_n \in \mathfrak{c}$ satisfy $[\nu_i] = [\nu]$ for all $i$.
	
	Denote the full subcategory of $\cate{Adm}$ so obtained by $\cate{Adm}_{[\nu]}$.
\end{definition}

The following is justified by earlier discussions, and is familiar to experts.

\begin{proposition}\label{prop:Adm-decomp}
	Each $\cate{Adm}_{[\nu]}$ is a Serre subcategory of $\cate{Adm}$, and $\cate{Adm} = \bigoplus_{[\nu]} \cate{Adm}_{[\nu]}$.
\end{proposition}

Clearly, $\mathcal{G}_\nu$ belongs to $\cate{Adm}_{[\nu]}$. Conversely, for $M$ in $\cate{Adm}_{[\nu]}$, the $\nu_1, \ldots, \nu_n \in \mathfrak{c}$ chosen in Remark \ref{rem:HC-admissible} satisfy $[\nu_i] = [\nu]$ for all $i$. Thus $\{ \mathcal{G}_\mu : [\mu] = [\nu] \}$ generates $\cate{Adm}_{[\nu]}$.

\begin{definition}\label{def:W-regular}
	Let $\dot{\nu} \in \mathfrak{t}^*$. If $\Stab_{\widetilde{\mathrm{W}}}(\dot{\nu}) = \{1\}$, then $\dot{\nu}$ is said to be \emph{$\widetilde{\mathrm{W}}$-regular}.
	
	Regularity depends only on the $\widetilde{\mathrm{W}}$-orbit of $\dot{\nu}$, therefore we can also talk about the $\widetilde{\mathrm{W}}$-regularity of $\nu \in \mathfrak{c}$ or of its image $[\nu]$ in $\mathfrak{t}^* / \widetilde{\mathrm{W}}$.
\end{definition}

\begin{lemma}\label{prop:affine-regular}
	Pick any preimage $\dot{\nu} \in \mathfrak{t}^*$ of $\nu \in \mathfrak{c}$. Then $\nu$ is $\widetilde{\mathrm{W}}$-regular if and only if $w\dot{\nu} - \dot{\nu} \notin \mathbf{X}^*(T)$ for all $w \in \mathrm{W} \smallsetminus \{1\}$.
\end{lemma}
\begin{proof}
	Express $\tilde{w} \in \widetilde{\mathrm{W}}$ as $\mu \rtimes w$ where $\mu \in \mathbf{X}^*(T)$ and $w \in \mathrm{W}$. Being $\widetilde{\mathrm{W}}$-regular means that $w(\dot{\nu}) + \mu \neq \dot{\nu}$ whenever $(\mu, w) \neq (0, 1)$; note that inequality holds whenever $\mu \neq 0$ and $w = 1$, hence $\nu$ is $\widetilde{\mathrm{W}}$-regular if and only if $w\dot{\nu} - \dot{\nu} \notin \mathbf{X}^*(T)$ for all $w \in \mathrm{W} \smallsetminus \{1\}$.
\end{proof}

\section{Homological properties of \texorpdfstring{$\widehat{\Wh}$}{hat W}}\label{sec:homological}
In this section, we consider the completion $\widehat{\Wh}$ of the filtered ring $\Wh = \Wh_G$ associated with a connected reductive group $G$ with auxiliary data $B = TN$ and $\bpsi$.

\subsection{Auslander regularity and dimension}
For the notion of Auslander regularity for a Noetherian ring $A$, see \cite[Appendix A:IV, (1.10)]{Bj93}. Its definition involves the number
\begin{equation}\label{eqn:grade}
	j_A(M) := \inf \underbracket{\{ k \in \Z: \Ext^k_A(M, A) \neq 0 \}}_{\neq \emptyset \;\text{by \textit{loc.\ cit.} when $M \neq \{0\}$}}
\end{equation}
defined for all finitely generated nonzero $A$-module $M$, called the \emph{grade number} of $M$, which will be used later on. 

Denote by $\gldim(A)$ the global dimension of $A$ as defined in \cite[Chapter 7, Sections 1.8--1.11]{McR01}.

Denote by $\injdim(A)$ the left injective dimension of $A$ as a left $A$-module. Define the right injective dimension of $A$ similarly. When $A$ is Noetherian and both injective dimensions are finite, then they are equal by \cite[Lemma A]{Za69}, and we simply say $\injdim(A)$ is the injective dimension of $A$.

We apply these notions to $A = \widehat{\Wh}$, recalling that $\widehat{\Wh}$ is Noetherian.

\begin{proposition}\label{prop:Auslander-regular}
	The ring $\widehat{\Wh}$ is Auslander regular, and $j_{\widehat{\Wh}}(M) = j_{\gr\widehat{\Wh}}(\gr(M))$ for all finitely generated nonzero $\widehat{\Wh}$-module $M$ equipped with a good filtration, in which case $\gr(M) \neq \{0\}$.
\end{proposition}
\begin{proof}
	By Remark \ref{rem:Zariskian}, the filtered ring $\widehat{\Wh}$ is Zariskian. Since $\gr\widehat{\Wh} \simeq \CC[\mathfrak{Z}]$ and $\mathfrak{Z}$ is smooth affine, the assertion follows from \cite[Appendix A:IV, (4.15)]{Bj93}.
\end{proof}

\begin{proposition}\label{prop:gldim}
	Let $n := \dim T$. Then $\injdim \widehat{\Wh} = n = \gldim \widehat{\Wh}$.
\end{proposition}
\begin{proof}
	Remark \ref{rem:Z-dim} gives $\dim \mathfrak{Z} = 2n$. We contend that $\gldim \widehat{\Wh} \leq n$.
	
	First, $\gldim(\widehat{\Wh}) = \sup_M j_{\widehat{\Wh}}(M) = \sup_M j_{\gr\widehat{\Wh}}(\gr(M))$ by Auslander-regularity (Proposition \ref{prop:Auslander-regular}) and \cite[Appendix IV: (1.11)]{Bj93}, where $M$ ranges over finitely generated nonzero $\widehat{\Wh}$-modules and $\gr(M)$ is relative to any good filtration.
	
	Secondly, since $M \neq \{0\} \implies \gr(M) \neq \{0\}$ and $\CC[\mathfrak{Z}]$ is regular, standard results from commutative algebra \cite[Theorem D.4.4 (i)]{HTT08} give
	\[ \dim \gr(M) + j_{\gr(\widehat{\Wh})}(\gr(M)) = \gldim \CC[\mathfrak{Z}] = 2n; \]
	here $\dim \gr(M)$ is the Krull dimension of $\gr(\widehat{\Wh})/\mathrm{ann}(\gr(M))$.
	
	Theorem \ref{prop:integrability} implies $\dim \gr(M) \geq n$, and then $j_{\gr(\widehat{\Wh})}(\gr(M)) \leq n$, as asserted.
	
	The foregoing arguments are adapted from the case for $\mathscr{D}$-modules.

	Now we have $\injdim \widehat{\Wh} \leq \gldim \widehat{\Wh} \leq n$. It remains to show $n \leq \injdim \widehat{\Wh}$. Indeed,
	\[ \Ext^n_{\widehat{\Wh}}(\widehat{\mathcal{G}}_\nu, \widehat{\Wh}) \simeq \widehat{\mathcal{G}}_\nu^{\mathrm{r}} \neq 0, \]
	which in turn follows from Lemma \ref{prop:Ext-HC} and the version of Lemma \ref{prop:Gnu-nonzero} for right modules.
\end{proof}

Consequently, all the ring-theoretic and homological tools from \cite[Appendix A:IV]{Bj93} for the study of $\mathscr{D}$-modules (or $\mathscr{E}$-modules) are applicable to $\widehat{\Wh}$-modules as well.

\subsection{Duality}\label{subsec:duality}
Abbreviate the grade \eqref{eqn:grade} of a finitely generated $\widehat{\Wh}$-module $M$ by $j(M)$. Set $n := \dim T$.

\begin{definition}[\protect{\cite[Section 3]{Iwa97}}]
	\label{def:holonomic}
	A finitely generated $\widehat{\Wh}$-module $M$ is said to be \emph{holonomic} if either $j(M) = n$ or $M = 0$. There is an analogous notion for right $\widehat{\Wh}$-modules.
\end{definition}

For a finitely generated $\widehat{\Wh}$-module $M$, being holonomic\ means that $\RHom_{\widehat{\Wh}}(M, \widehat{\Wh})$ is concentrated at the maximal possible degree $n$. This is related to the usual notion in terms of dimensions of singular supports as follows.

\begin{proposition}\label{prop:admissible-hh}
	Let $M$ be a finitely generated nonzero $\widehat{\Wh}$-module. Then:
	\begin{enumerate}[(i)]
		\item $\dim \mathrm{SS}(M) \leq n$ if and only if $\dim \mathrm{SS}(M) = n$, if and only if $M$ is holonomic;
		\item if $M = \widehat{N}$ for some admissible $\Wh$-module $N$, then $M$ is holonomic.
	\end{enumerate}
\end{proposition}
\begin{proof}
	Take any good filtration on $M$. Proposition \ref{prop:Auslander-regular} gives $j(M) = j_{\gr(\widehat{\Wh})}(\gr(M))$, and it is shown in the proof of Proposition \ref{prop:gldim} that
	\begin{align*}
		\dim \gr(M) + j_{\widehat{\Wh}}(M) & = 2n, \\
		\dim \mathrm{SS}(M) = \dim \gr(M) & \geq n.
	\end{align*}
	The assertion (i) follows at once.
	
	For (ii), since $\mathrm{SS}(M) = \mathrm{SS}(N)$ by Proposition \ref{prop:SS-completion}, it suffices to apply Lemma \ref{prop:SS-bound} and note that the fibers of $\kappa: \mathfrak{Z} \to \mathfrak{c}$ are all Lagrangian.
\end{proof}

\begin{theorem}[Y.\ Iwanaga \protect{\cite{Iwa97}}]\label{prop:Iwanaga}
	The holonomic $\widehat{\Wh}$-modules have the following properties.
	\begin{enumerate}[(i)]
		\item Every holonomic module has finite length.
		\item Holonomic modules form a Serre subcategory of $\widehat{\Wh}\dcate{Mod}$.
		\item There is an equivalence of Abelian categories
		\[\begin{tikzcd}[row sep=tiny]
			\left\{ \text{holonomic left $\widehat{\Wh}$-modules} \right\} \arrow[shift left, r] & \left\{ \text{holonomic right $\widehat{\Wh}$-modules} \right\} \arrow[shift left, l] \\
			M \arrow[mapsto, r] & M^\vee \\
			{}^\vee M & M \arrow[mapsto, l]
		\end{tikzcd}\]
		where $M^\vee$ and ${}^\vee M$ are both defined by the expression $\Ext^n_{\widehat{\Wh}}(M, \widehat{\Wh})$.
	\end{enumerate}
\end{theorem}
\begin{proof}
	This is an application of the Theorems 6---8 of \cite{Iwa97}, by noting that $\injdim \widehat{\Wh} = n$.
\end{proof}

Given a finitely generated $\widehat{\Wh}$-module $M$ (resp.\ finitely generated $\Wh$-module $N$), recall that $\Ext^i_{\widehat{\Wh}}(M, \widehat{\Wh})$ (resp.\ $\Ext^i_{\Wh}(N, \Wh)$) is finitely generated over $\widehat{\Wh}$ (resp.\ $\Wh$) for all $i$. Hence one may consider $\Ext_{\Wh}^i(N, \Wh)^{\wedge}$ with respect to any good filtration.

The $\Ext$-modules of a finitely generated $\Wh$-module $N$ can be linked to those of $M := \widehat{N}$ by the observation below.

\begin{lemma}\label{prop:Ext-completion}
	Let $N$ be a finitely generated $\Wh$-module, then $\Ext^i_{\widehat{\Wh}}(\widehat{N}, \widehat{\Wh}) \simeq \Ext^i_{\Wh}(N, \Wh) \dotimes{\Wh} \widehat{\Wh} \simeq \Ext^i_{\Wh}(N, \Wh)^{\wedge}$ as right $\Wh$-modules, for all $i$.
\end{lemma}
\begin{proof}
	Recall that $\Wh \to \widehat{\Wh}$ is flat as left and right modules, and $\widehat{N} \simeq \widehat{\Wh} \dotimes{\Wh} N$. Similarly, $L \dotimes{\Wh} \widehat{\Wh} \simeq \widehat{L}$ for every finitely generated right $\Wh$-module $L$. The assertion follows by taking a free resolution for $N$. See also \cite[Lemma 7.15]{BNS24} in the case $X = \Wh \subset \widehat{\Wh} = Y$.
\end{proof}

Admissibility is also defined for right $\Wh$-modules. The next result says that $\Ext_{\Wh}^n(\cdot, \Wh)$ preserves admissibility, up to completions and filtrations.

\begin{proposition}\label{prop:Ext-admissible}
	Let $N$ be an admissible $\Wh$-module, then $\Ext_{\Wh}^n(N, \Wh)^{\wedge}$ is a successive extension of right $\widehat{\Wh}$-modules which are completions of admissible right $\Wh$-modules.
\end{proposition}
\begin{proof}
	For all finitely generated $\Wh$-module $N'$, set $E(N') := \Ext_{\Wh}^n(N', \Wh)$ and $\widehat{E}(N') := E(N')^{\wedge}$. Theorem \ref{prop:Iwanaga} and Lemma \ref{prop:Ext-completion} imply the exactness of $\widehat{E}: \Wh\dcate{Mod}_{\mathrm{fg}}^{\mathrm{opp}} \to \cated{Mod}\widehat{\Wh}_{\mathrm{fg}}$.
	
	Now apply Remark \ref{rem:HC-admissible} to $N$ to obtain a filtration
	\[ \{0\} = N_0 \subset \cdots \subset N_n = N \]
	together with surjections $\mathcal{G}_{\nu_i} \twoheadrightarrow N_i/N_{i-1}$ for all $i$. Utilizing the exactness of $\widehat{E}$, we infer that $\widehat{E}(N)$ admits a filtration of length $n$, whose subquotients come with embeddings
	\[ \left[ \widehat{E}(N_i/N_{i-1}) \hookrightarrow \widehat{\mathcal{G}}_{\nu_i}^{\mathrm{r}} \right] = \left[ E(N_i/N_{i-1}) \to \mathcal{G}_{\nu_i}^{\mathrm{r}} \right]^{\wedge} \]
	by Lemma \ref{prop:Ext-HC}. Hence the subquotients are completions of admissible right $\Wh$-modules.
\end{proof}

\section{The case with \texorpdfstring{$\widetilde{\mathrm{W}}$}{tilde-W}-regular infinitesimal characters}\label{sec:regular-inf-char}
\subsection{Localizations of nil-DAHA}
In this subsection, $G$ is simply connected and equipped with a Borel pair $(B, T)$. We make use of the formalism of degenerate nil-DAHA in Subsection \ref{subsec:DAHA}.

Define $\widetilde{\mathrm{W}} = \mathrm{W} \ltimes \mathbf{X}^*(T)$ as in Subsection \ref{subsec:inf-character}. Let $D(T)_\hbar$ be the Rees algebra of $D(T)$ relative to the order filtration. By \cite[(7.2.2)]{Gin18} (replacing the $\Hk$ in \textit{loc.\ cit.} by $\mathscr{H}(\mathfrak{t}_{\mathrm{aff}}, \widetilde{\mathrm{W}})$), there is an embedding of $\CC[\hbar]$-algebras
\begin{equation}\label{eqn:WD-embedding-1}
	\mathrm{W} \ltimes D(T)_\hbar \hookrightarrow \mathscr{H}(\mathfrak{t}_{\mathrm{aff}}, \widetilde{\mathrm{W}}).
\end{equation}
Make $\mathrm{W} \ltimes D(T)_\hbar$ into a graded $\CC[\hbar]$-algebra by placing $\mathrm{W}$ at degree zero, then \eqref{eqn:WD-embedding-1} is a graded embedding, $\mathrm{W} \ltimes D(T)_\hbar$ equals the Rees algebra of $\mathrm{W} \ltimes D(T)$, and $\Sym(\mathfrak{t}_{\mathrm{aff}}) = \Sym(\mathfrak{t})[\hbar]$ embeds in $D(T)_\hbar$ via invariant vector fields on $T$.

Let $\Phi(G, T)$ be the set of all roots.

\begin{lemma}\label{prop:R-Ore}
	Let $\mathcal{R}$ be the multiplicative subset of $\Sym(\mathfrak{t}_{\mathrm{aff}}) \subset D(T)_\hbar$ generated by all
	\[ \check{\alpha} - c\hbar, \quad \alpha \in \Phi(G, T), \quad c \in \Z. \]
	Then $\mathcal{R}$ contains no zero-divisors of $D(T)_\hbar$ on either side, and $\mathcal{R}$ is an Ore set in $D(T)_\hbar$ (see Subsection \ref{subsec:Ore-sets}).
\end{lemma}
\begin{proof}
	To show $\mathcal{R}$ contains no zero-divisors on either side, note that $D(T)$ has no zero-divisors on either side except zero, and then apply the first part of Proposition \ref{prop:Ore-Rees}.
	
	Next, $\mathcal{R}$ commutes with $\Sym(\mathfrak{t}_{\mathrm{aff}})$, hence it remains to check the Ore conditions relative to elements $\mu \in \mathbf{X}^*(T)$. This is settled by \cite[(7.2.1)]{Gin18} by taking $\xi = \check{\alpha}$ in \text{loc.\ cit.}
\end{proof}

In \textit{loc.\ cit.}, the Ore condition is established for the larger subset $\mathcal{S} := \Sym(\mathfrak{t}_{\mathrm{aff}}) \smallsetminus \{0\}$.

\begin{remark}
	The generators of $\mathcal{R}$ are all homogeneous of degree $1$. This makes the Ore localization $\mathcal{R}^{-1} D(T)_\hbar$ of $D(T)_\hbar$ (see Proposition \ref{prop:Ore-Rees}) into a $\Z$-graded $\CC[\hbar]$-algebra. Moreover, $\mathrm{W}$ acts on $\mathcal{R}^{-1} D(T)_\hbar$.
\end{remark}

\begin{lemma}\label{prop:WD-embeddings}
	There is an embedding $\mathscr{H}(\mathfrak{t}_{\mathrm{aff}}, \widetilde{\mathrm{W}}) \hookrightarrow \mathrm{W} \ltimes \mathcal{R}^{-1} D(T)_\hbar$ of graded $\CC[\hbar]$-algebras, whose composition with \eqref{eqn:WD-embedding-1} is the natural homomorphism $\mathrm{W} \ltimes D(T)_\hbar \to \mathrm{W} \ltimes \mathcal{R}^{-1} D(T)_\hbar$.
\end{lemma}
\begin{proof}
	This is similar to \cite[(7.2.2)]{Gin18}, by noting that to define Demazure operators in $\mathscr{H}(\mathfrak{t}_{\mathrm{aff}}, \widetilde{\mathrm{W}})$, one only needs $\check{\alpha}$ in the denominators where $\alpha \in \Sigma_{\mathrm{aff}}$, and these are all captured by $\mathcal{R}$; for a description in the case where $\alpha \in \Sigma_{\mathrm{aff}}$ is the simple affine reflection, we refer to \cite[(7.2.5)]{Gin18}.
\end{proof}

\begin{definition}
	Let $\mathcal{U}_0$ be the multiplicative subset of $\Sym(\mathfrak{t})$ generated by $\check{\alpha} - c$, where $\alpha \in \Phi(G, T)$ and $c \in \Z$. Let $\mathcal{U} := \mathcal{U}_0 \cap \Sym(\mathfrak{t})^{\mathrm{W}}$.
\end{definition}

Since $\mathcal{U}_0$ is obtained from $\mathcal{R}$ by specializing at $\hbar = 1$, it is an Ore set in $D(T)$ by Proposition \ref{prop:Ore-Rees}.

Proposition \ref{prop:Ore-Rees} also implies that all the three graded $\CC[\hbar]$-algebras in Lemma \ref{prop:WD-embeddings} arise as Rees algebras, and specialization at $\hbar = 1$ yields embeddings of filtered algebras
\begin{equation}\label{eqn:WD-embedding-2}
	\mathrm{W} \ltimes D(T) \hookrightarrow \Hk \hookrightarrow  \mathrm{W} \ltimes \mathcal{U}_0^{-1} D(T).
\end{equation}

\begin{proposition}\label{prop:Hsph-sandwich}
	With the notations above, $\mathcal{U}$ is an Ore set in $D(T)^{\mathrm{W}}$, and there are canonical embeddings of filtered $\CC$-algebras
	\begin{equation*}
		D(T)^{\mathrm{W}} \hookrightarrow \Hk^{\mathrm{sph}} \hookrightarrow \mathcal{U}^{-1} D(T)^{\mathrm{W}}.
	\end{equation*}
	Their composition is the natural homomorphism $D(T)^{\mathrm{W}} \to \mathcal{U}^{-1} D(T)^{\mathrm{W}}$.
\end{proposition}
\begin{proof}
	Let $P \in D(T)^{\mathrm{W}}$ and $u \in \mathcal{U}$ be given. To check the right Ore condition for $\mathcal{U}$, we seek $P' \in D(T)^{\mathrm{W}}$ and $u' \in \mathcal{U}$ such that $Pu' = uP'$.
	
	Note that $\mathcal{U}_0$ is $\mathrm{W}$-stable. By the right Ore condition satisfied by $\mathcal{U}_0$ in $D(T)$, there exists $Q' \in D(T)$ and $v' \in \mathcal{U}_0$ such that $Pv' = uQ'$. Let $v_1 \in \mathcal{U}_0$ be the product of all $w$-translates of $v'$ where $w \in \mathrm{W} \smallsetminus \{1\}$, so that $u' := v' v_1 \in \mathcal{U}$. Then $P u' = u Q' v_1$; the left hand side is $\mathrm{W}$-invariant, hence so is $u Q' v_1$. As $u$ is $\mathrm{W}$-invariant, and nonzero elements of $D(T)$ are never zero-divisors on either side (cf.\ the proof of Proposition \ref{prop:regular-elements}), we infer that $P' := Q' v_1 \in D(T)^{\mathrm{W}}$.
	
	This settles the right Ore condition. The left case is similar.
	
	To obtain the displayed embeddings, multiply \eqref{eqn:WD-embedding-2} by $\mathbf{e}$ on the left and right. The middle term is $\Hk^{\mathrm{sph}}$. The left term is $(\mathbf{e} \ltimes D(T)) \mathbf{e} = \mathbf{e} \ltimes D(T)^{\mathrm{W}}$ since $\mathbf{e}w = \mathbf{e} = w\mathbf{e}$ for all $w \in \mathrm{W}$. Similarly, the right term is $\mathbf{e} \ltimes (\mathcal{U}_0^{-1} D(T))^{\mathrm{W}}$, and we claim that
	\[ (\mathcal{U}_0^{-1} D(T))^{\mathrm{W}} = \mathcal{U}^{-1} D(T)^{\mathrm{W}} \quad \text{inside}\; \mathcal{U}^{-1} D(T). \]
	It suffices to show $\subset$. By the left Ore condition of $\mathcal{U}_0$, the elements of $\mathcal{U}_0^{-1} D(T)$ are expressed formally as fractions $u^{-1} P$. Suppose $u^{-1} P$ is $\mathrm{W}$-invariant. As in the first step of the proof, one can choose $u_1 \in \mathcal{U}_0$ such that $v := u_1 u \in \mathcal{U}$. Therefore $u^{-1} P = v^{-1} u_1 P$, and we see as before that $Q := u_1 P \in D(T)^{\mathrm{W}}$. Hence $u^{-1} P = v^{-1} Q \in \mathcal{U}^{-1} D(T)^{\mathrm{W}}$ and the claim follows.
	
	The desired embeddings follow at once. The assertion about their composition results from Lemma \ref{prop:WD-embeddings}.
\end{proof}

\subsection{Reductive case}\label{subsec:reductive-case}
Hereafter, we consider a connected reductive group $G$ with chosen Borel pair $(B, T)$ where $B = TN$, an adapted principal $\mathfrak{sl}(2)$-triple $(\mathsf{e}, \mathsf{h}, \mathsf{f})$ where $\mathsf{e} \in \mathfrak{n}^-$, and an invariant symmetric form $(\cdot, \cdot)$ on $\mathfrak{g}$ with $\bpsi = (\mathsf{e}, \cdot): \mathfrak{n} \to \CC$. Write $\mathrm{W} = \mathrm{W}(G, T)$.

Let $T_{\mathrm{SC}}$ be the preimage of $T$ in $G_{\mathrm{SC}}$, the simply connected cover of $G_{\mathrm{der}}$. Take the isogeny $G' = G_{\mathrm{SC}} \times A_G \to G$ in \eqref{eqn:G-SES} with kernel $Z$, and set
\[ T' := T_{\mathrm{SC}} \times A_G. \]
Therefore
\begin{gather*}
	D(T') = D(T_{\mathrm{SC}}) \otimes D(A_G), \quad D(T')^{\mathrm{W}} = D(T_{\mathrm{SC}})^{\mathrm{W}} \otimes D(A_G).
\end{gather*}

The $Z$-action on $D(T)$ induced by translations preserves ring structure and commutes with $\mathrm{W}$. Thus
\begin{gather*}
	D(T) \simeq D(T')^{Z\text{-inv}}, \quad D(T)^{\mathrm{W}} \simeq D(T')^{\mathrm{W}, Z\text{-inv}}.
\end{gather*}

Let $\mathcal{U}_0$ (resp.\ $\mathcal{U}$) be the Ore set in $D(T_{\mathrm{SC}})$ (resp.\ $D(T_{\mathrm{SC}})^{\mathrm{W}}$) defined for $G_{\mathrm{SC}}$. Their elements are all $Z$-invariant, hence $\mathcal{U}_0$ (resp.\ $\mathcal{U}$) is also a multiplicative subset of $D(T)$ (resp.\ $D(T)^{\mathrm{W}}$), containing no zero-divisors on either side (cf.\ the proof of Lemma \ref{prop:R-Ore}).

\begin{proposition}\label{prop:W-sandwich}
	With the notations above, $\mathcal{U}$ is an Ore set in $D(T)^{\mathrm{W}}$, and there are canonical embeddings of filtered $\CC$-algebras
	\begin{equation*}
		D(T)^{\mathrm{W}} \hookrightarrow \Wh \hookrightarrow \mathcal{U}^{-1} D(T)^{\mathrm{W}}
	\end{equation*}
	commuting with the homomorphisms from $\Sym(\mathfrak{t})^{\mathrm{W}} \simeq \mathcal{Z}(\mathfrak{g})$ to each term. They compose to the natural homomorphism $D(T)^{\mathrm{W}} \to \mathcal{U}^{-1} D(T)^{\mathrm{W}}$.
\end{proposition}
\begin{proof}
	As $\mathcal{U}_0$ (resp.\ $\mathcal{U}$) is an Ore set in $D(T_{\mathrm{SC}})$ (resp.\ $D(T_{\mathrm{SC}})^{\mathrm{W}}$), so is it inside $D(T')$ (resp.\ $D(T')^{\mathrm{W}}$). Since the $Z$-action preserves ring structure, the left and right Ore conditions of $\mathcal{U}_0$ (resp.\ $\mathcal{U}$) in $D(T)$ (resp.\ $D(T)^{\mathrm{W}}$) follow from those in $D(T')$ (resp.\ $D(T')^{\mathrm{W}}$) by extracting $Z$-invariant parts.
	
	Identify $\Wh_{G_{\mathrm{SC}}}$ with $\Hk^{\mathrm{sph}}$. Then the embeddings of filtered algebras for the $G_{\mathrm{SC}}$ case reduces to Proposition \ref{prop:Hsph-sandwich}; commutation with homomorphisms from $\Sym(\mathfrak{t}_{\mathrm{SC}})^{\mathrm{W}}$ is immediate. The case for $G'$ follows by $(\cdot) \otimes D(A_G)$, since localization with respect to $\mathcal{U}$ has no effect on the factor $D(A_G)$. Finally, the case for $G$ follows by taking $Z$-invariants.
\end{proof}

\begin{proposition}\label{prop:DT-Morita}
	Let $(D(T), \mathrm{W})\dcate{Mod}$ be the category of $\mathrm{W}$-equivariant $D(T)$-modules. There is an equivalence
	\[ D(T) \dotimes{D(T)^{\mathrm{W}}} (\cdot): D(T)^{\mathrm{W}}\dcate{Mod} \rightiso (D(T), \mathrm{W})\dcate{Mod}, \]
	where $\mathrm{W}$ acts on the tensor slot $D(T)$.
\end{proposition}
\begin{proof}
	Make $D(T)$ into a $\left( \mathrm{W} \ltimes D(T), D(T)^{\mathrm{W}} \right)$-bimodule. Then it induces a Morita equivalence $D(T)^{\mathrm{W}}\dcate{Mod} \rightiso \mathrm{W} \ltimes D(T)\dcate{Mod}$. This is indeed the content of \cite[Chapter 7, Sections 8.3--8.6]{McR01}, the inputs being the simplicities of
	\begin{itemize}
		\item the ring $D(T)$, see \cite[Chapter 15, Theorem 3.8]{McR01};
		\item the smash product $\mathrm{W} \ltimes D(T)$, following from the previous item and \cite[Chapter 7, Proposition 8.12]{McR01} since $\mathrm{W} \smallsetminus \{1\}$ acts on $D(T)$ by non-inner automorphisms.
	\end{itemize}
	
	Finally, there is an evident equivalence $\mathrm{W} \ltimes D(T)\dcate{Mod} \simeq (D(T), \mathrm{W})\dcate{Mod}$ that commutes with forgetful functors to $D(T)\dcate{Mod}$.
\end{proof}

Combining Propositions \ref{prop:W-sandwich} and \ref{prop:DT-Morita}, we obtain functors
\begin{equation}\label{eqn:Morita-T}
	\Wh\dcate{Mod} \to D(T)^{\mathrm{W}}\dcate{Mod} \rightiso (D(T), \mathrm{W})\dcate{Mod}.
\end{equation}

\begin{remark}
	Longergan \cite[(1.2.2) and (1.2.3)]{Lo18} proved that \eqref{eqn:Morita-T} identifies $\Wh\dcate{Mod}$ with a precisely defined full subcategory of $(D(T), \mathrm{W})\dcate{Mod}$. See also \cite[Theorem 1.5.1]{Gin18} for the case of holonomic modules and simply connected $G$. Although these deep results can be used to prove the upcoming results in Subsections \ref{subsec:HC-W-regular} and \ref{subsec:decompletion-W-regular}, we opt for a more direct and elementary approach in what follows.
\end{remark}

\subsection{Non-integral parameters}
Retain the conventions in Subsection \ref{subsec:reductive-case}, and recall that $\widetilde{\mathrm{W}} := \mathrm{W} \ltimes \mathbf{X}^*(T)$ acts on $\mathfrak{t}^*$ by affine transformations.

\begin{definition}\label{def:integral-nu}
	Let $\nu \in \mathfrak{c}$ with a preimage $\dot{\nu} \in \mathfrak{t}^*$. If $\lrangle{\dot{\nu}, \check{\alpha}} \notin \Z$ for all coroots $\check{\alpha}$ for $T_{\mathrm{SC}} \subset G_{\mathrm{SC}}$, identified with their images in $\mathbf{X}_*(T)$, then $\nu$ is said to be \emph{non-integral}.
\end{definition}

The conditions above depend only on the $\widetilde{\mathrm{W}}$-orbit of $\dot{\nu}$, hence the non-integrality depends only on the class $[\nu]$ in the quotient set $\mathfrak{t}^* / \widetilde{\mathrm{W}}$.

\begin{lemma}\label{prop:nonintegral}
	If $[\nu]$ is $\widetilde{\mathrm{W}}$-regular (Definition \ref{def:W-regular}), then it is non-integral.
\end{lemma}
\begin{proof}
	Pick any representative $\dot{\nu} \in \mathfrak{t}^*$. Let $\check{\alpha}$ be any coroot for $T_{\mathrm{SC}} \subset G_{\mathrm{SC}}$ and $c \in \Z$. The affine transformation $\tilde{w} = s_\alpha + c\alpha$ on $\mathfrak{t}^*$ lies in $\widetilde{\mathrm{W}}$, and $\tilde{w}\dot{\nu} = \dot{\nu}$ if and only if $\lrangle{\dot{\nu}, \check{\alpha}} = c$.
\end{proof}

In what follows, $\Sym(\mathfrak{t})^{\mathrm{W}}$, $\mathcal{Z}(\mathfrak{g})$ and $\CC[\mathfrak{c}]$ are identified. For every $\nu \in \mathfrak{c}$ (resp.\ $\dot{\nu} \in \mathfrak{t}^*$), denote by $\mathfrak{m}_\nu \subset \Sym(\mathfrak{t})^{\mathrm{W}}$ (resp.\ $\mathfrak{m}_{\dot{\nu}} \subset \Sym(\mathfrak{t})$) the corresponding maximal ideal.

\begin{lemma}\label{prop:DT-eigenvalue}
	Let $N$ be a $D(T)$-module. Suppose that there exist
	\begin{itemize}
		\item a family $(\dot{\nu}_i)_{i \in I}$ of elements of $\mathfrak{t}^*$,
		\item generators $(x_i)_{i \in I}$ of $N$ such that $V_i := \Sym(\mathfrak{t}) x_i$ is finite-dimensional for each $i \in I$; denote the set of $\mathfrak{t}$-eigenvalues for $V_i$ by $E_i \subset \mathfrak{t}^*$.
	\end{itemize}
	Then $N$ is locally finite as a $\mathfrak{t}$-module, and each $\mathfrak{t}$-eigenvalue for $N$ lies in $\mathbf{X}^*(T) + E_i$ for some $i \in I$.
\end{lemma}
\begin{proof}
	Fix $i \in I$. The same arguments as in Lemma \ref{prop:admissible-generation} realizes the $\mathfrak{t}$-module $D(T)x_i$ as a quotient of $D(T) \otimes V_i$, where $\mathfrak{t}$ acts on $D(T)$ by $\xi \cdot P = [\xi, P]$ and $D(T) \otimes V_i$ is the tensor product $\mathfrak{t}$-module. Moreover $D(T)$ upgrades to a $T$-module. We infer that $D(T)x_i$ is locally $\mathfrak{t}$-finite with eigenvalues in $\mathbf{X}^*(T) + E_i$, and conclude by varying $i$.
\end{proof}

Let $\mathcal{U}$ and $\mathcal{U}_0$ be as in Subsection \ref{subsec:reductive-case}.

\begin{lemma}\label{prop:U-invertible}
	Let $L$ be a $D(T)^{\mathrm{W}}$-module, hence also a $\Sym(\mathfrak{t})^{\mathrm{W}}$-module. If there exist
	\begin{itemize}
		\item a family $(\nu_i)_{i \in I}$ of non-integral elements of $\mathfrak{c}$,
		\item generators $(y_i)_{i \in I}$ of $L$ such that each $y_i$ is annihilated by some power of $\mathfrak{m}_{\nu_i}$,
	\end{itemize}
	then $\mathcal{U}$ acts invertibly on $L$.
\end{lemma}
\begin{proof}
	Observe that every $\mathrm{W}$-isotypic component of $D(T)$ is a left (resp.\ right) $D(T)^{\mathrm{W}}$-module. It follows that $D(T)^{\mathrm{W}}$ is a direct summand of $D(T)$ where $D(T)$ is viewed as a left (resp.\ right) $D(T)^{\mathrm{W}}$-module.
	
	Therefore, $L$ embeds into the $D(T)$-module $N := D(T) \dotimes{D(T)^{\mathrm{W}}} L$, which is generated by all $x_i := 1 \otimes y_i$. It suffices to show that $\mathcal{U}_0$ acts invertibly on $N$.
	
	Fix $i \in I$ and pick any preimage $\dot{\nu}_i \in \mathfrak{t}^*$ of $\nu_i$. Note that $V_i := \Sym(\mathfrak{t}) x_i$ is finite-dimensional with $\mathfrak{t}$-eigenvalues lying in $\mathrm{W} \dot{\nu}_i$. Lemma \ref{prop:DT-eigenvalue} now implies $N$ is locally $\mathfrak{t}$-finite with eigenvalues in $\bigcup_i (\mathbf{X}^*(T) + \mathrm{W} \dot{\nu}_i)$.
	
	Elements in $\mathbf{X}^*(T) + \mathrm{W}\dot{\nu}_i$ pair non-integrally with every coroot for $T_{\mathrm{SC}} \subset G_{\mathrm{SC}}$, hence $\mathcal{U}_0$ acts invertibly on $N$ as desired.
\end{proof}

\begin{proposition}\label{prop:non-integral-isom}
	For non-integral $\nu \in \mathfrak{c}$, there is a homomorphism
	\[ \rho: D(T)^{\mathrm{W}} / D(T)^{\mathrm{W}}\mathfrak{m}_\nu \to \Wh/\Wh \mathfrak{m}_\nu \]
	of $D(T)^{\mathrm{W}}$-modules induced by those in Proposition \ref{prop:W-sandwich}; in fact, $\rho$ is an isomorphism.
\end{proposition}
\begin{proof}
	Consider the embeddings of filtered algebras in Proposition \ref{prop:W-sandwich}. They commute with homomorphisms from $\CC[\mathfrak{c}]$, whence the homomorphisms of $D(T)^{\mathrm{W}}$-modules
	\[ D(T)^{\mathrm{W}} \big/ D(T)^{\mathrm{W}} \mathfrak{m}_\nu \xrightarrow{\rho} \Wh/\Wh\mathfrak{m}_\nu \xrightarrow{\eta} \mathcal{U}^{-1} D(T)^{\mathrm{W}} \big/ \mathcal{U}^{-1} D(T)^{\mathrm{W}} \mathfrak{m}_\nu . \]
	
	We contend that $\rho$ is surjective. Given $w \in \Wh$, apply the right Ore condition of $\mathcal{U}$ to get $w = Pu^{-1}$ where $P \in D(T)^{\mathrm{W}}$ and $u \in \mathcal{U} \subset \Sym(\mathfrak{t})^{\mathrm{W}}$. Hence $wu \in \Image[D(T)^{\mathrm{W}} \to \Wh]$.
	
	Express $u$ as a $\mathrm{W}$-invariant product $\prod_{j=1}^m (\check{\alpha}_j - c_j)$ where $\check{\alpha}_j$ are coroots for $G_{\mathrm{SC}}$ and $c_j \in \Z$. Taking a preimage $\dot{\nu} \in \mathfrak{t}^*$ of $\nu$, we see in $\Sym(\mathfrak{t})^{\mathrm{W}}$ that $u$ is congruent modulo $\mathfrak{m}_\nu$ to
	\[ c := \prod_{j=1}^m \left( \lrangle{\dot{\nu}, \check{\alpha}_j} - c_j \right) \in \CC^{\times}. \]
	This implies $u(1 + \Wh\mathfrak{m}_\nu) = c + \Wh\mathfrak{m}_\nu$, and then $w + \Wh\mathfrak{m}_\nu = c^{-1} (wu + \Wh\mathfrak{m}_\nu) \in \Image(\rho)$ as required.
	
	It remains to show $\rho$ is injective. We contend that
	\[ \eta\rho: D(T)^{\mathrm{W}} \big/ D(T)^{\mathrm{W}} \mathfrak{m}_\nu \to \mathcal{U}^{-1} D(T)^{\mathrm{W}} \big/ \mathcal{U}^{-1} D(T)^{\mathrm{W}} \mathfrak{m}_\nu \]
	is injective. If $P + D(T)^{\mathrm{W}}\mathfrak{m}_\nu \in \Ker(\eta\rho)$, then by the construction of Ore localization via left fractions, there exists $u \in \mathcal{U}$ such that $uP \in D(T)^{\mathrm{W}} \mathfrak{m}_\nu$; equivalently, $uP$ annihilates the element $1 + D(T)^{\mathrm{W}} \mathfrak{m}_\nu$ of $D(T)^{\mathrm{W}}/D(T)^{\mathrm{W}}\mathfrak{m}_\nu$.
	
	The cyclic $D(T)^{\mathrm{W}}$-module $D(T)^{\mathrm{W}}/D(T)^{\mathrm{W}}\mathfrak{m}_\nu$ meets the premises of Lemma \ref{prop:U-invertible}, thus $u$ acts invertibly. So $P$ also annihilates $1 + D(T)^{\mathrm{W}} \mathfrak{m}_\nu$; equivalently, $P \in D(T)^{\mathrm{W}} \mathfrak{m}_\nu$ as desired.
\end{proof}

When viewed as a $\Wh$-module, $\Wh/\Wh \mathfrak{m}_\nu$ is the Harish-Chandra module $\mathcal{G}_\nu$ (Definition \ref{def:HC-module}).

\subsection{Harish-Chandra modules with \texorpdfstring{$\widetilde{\mathrm{W}}$}{W}-regular parameters}
\label{subsec:HC-W-regular}
We are going to apply the prior results to study $\mathcal{G}_\nu$ for $\widetilde{\mathrm{W}}$-regular $\nu$ (Definition \ref{def:W-regular}).

\begin{theorem}\label{prop:simplicity}
	For $\widetilde{\mathrm{W}}$-regular $\nu$, the Harish-Chandra $\Wh$-module $\mathcal{G}_\nu$ is simple.
\end{theorem}
\begin{proof}
	Note that $\nu$ is non-integral by Lemma \ref{prop:nonintegral}, and $\nu \in \mathfrak{c}_{\mathrm{reg}}$. Hence
	\begin{equation}\label{eqn:simplicity-aux0}
		\Sym(\mathfrak{t}) \big/ \Sym(\mathfrak{t})\mathfrak{m}_\nu \simeq \prod_{\dot{\nu} \mapsto \nu} \Sym(\mathfrak{t}) \big/ \Sym(\mathfrak{t}) \mathfrak{m}_{\dot{\nu}}
	\end{equation}
	as $\Sym(\mathfrak{t})$-algebras, where $\dot{\nu} \in \mathfrak{t}^*$ ranges over the preimages of $\nu$. These preimages form a $\mathrm{W}$-torsor in $\mathfrak{t}^*_{\mathrm{reg}}$.
	
	In view of Proposition \ref{prop:non-integral-isom}, it suffices to show $D(T)^{\mathrm{W}} / D(T)^{\mathrm{W}} \mathfrak{m}_\nu$ is simple as a $D(T)^{\mathrm{W}}$-module.	By \eqref{eqn:Morita-T}, this is equivalent to the simplicity of
	\begin{equation*}
		D(T) \dotimes{D(T)^{\mathrm{W}}} \left( D(T)^{\mathrm{W}} \big/ D(T)^{\mathrm{W}} \mathfrak{m}_\nu \right) \simeq D(T) \dotimes{D(T)^{\mathrm{W}}} \left( D(T)^{\mathrm{W}} \dotimes{\Sym(\mathfrak{t})^{\mathrm{W}}} \left(\Sym(\mathfrak{t})^{\mathrm{W}} / \mathfrak{m}_\nu \right) \right)
	\end{equation*}
	as an object of $(D(T), \mathrm{W})\dcate{Mod}$ or $(W \ltimes D(T)) \dcate{Mod}$. Using the commutative diagram of algebras
	\[\begin{tikzcd}
		\Sym(\mathfrak{t})^{\mathrm{W}} \arrow[r] \arrow[d] & \Sym(\mathfrak{t}) \arrow[d] \\
		D(T)^{\mathrm{W}} \arrow[r] & D(T)
	\end{tikzcd}\]
	the aforementioned object is isomorphic to
	\begin{multline*}
	D(T) \dotimes{\Sym(\mathfrak{t})} \left( \Sym(\mathfrak{t}) \dotimes{\Sym(\mathfrak{t})^{\mathrm{W}}} \left( \Sym(\mathfrak{t})^{\mathrm{W}} / \mathfrak{m}_\nu \right)\right) \simeq D(T) \dotimes{\Sym(\mathfrak{t})} \left( \Sym(\mathfrak{t}) \big/ \Sym(\mathfrak{t}) \mathfrak{m}_\nu \right) \\
	\stackrel{\text{\eqref{eqn:simplicity-aux0}}}{\simeq} \bigoplus_{\dot{\nu} \mapsto \nu} \left( D(T) \dotimes{\Sym(\mathfrak{t})} \left( \Sym(\mathfrak{t}) / \mathfrak{m}_{\dot{\nu}} \right)\right) \simeq \bigoplus_{\dot{\nu} \mapsto \nu} D(T) \big/ D(T) \mathfrak{m}_{\dot{\nu}}.
	\end{multline*}
	In the rightmost expression, $\mathrm{W}$ acts on $D(T)$ and permutes the summands.
	
	For every $\dot{\nu} \mapsto \nu$, the $D(T)$-module $D(T) \big/ D(T) \mathfrak{m}_{\dot{\nu}}$ is a flat connection on $T$; it is simple since the corresponding vector bundle is readily seen to be of rank $1$.
	
	Lemma \ref{prop:DT-eigenvalue} implies $D(T) \big/ D(T) \mathfrak{m}_{\dot{\nu}}$ is locally $\mathfrak{t}$-finite with $\mathfrak{t}$-eigenvalues included in $\mathbf{X}^*(T) + \dot{\nu} \subset \mathfrak{t}^*$. Given $w \in \mathrm{W} \smallsetminus \{1\}$, the $\mathfrak{t}$-eigenvalues in $D(T) \big/ D(T) \mathfrak{m}_{\dot{\nu}}$ and $D(T) \big/ D(T) \mathfrak{m}_{w\dot{\nu}}$ are disjoint, since otherwise one would get $w\dot{\nu} - \dot{\nu} \in \mathbf{X}^*(T)$, contradicting the $\widetilde{\mathrm{W}}$-regularity of $\nu$ and Lemma \ref{prop:affine-regular}.
	
	All in all, $D(T) \big/ D(T) \mathfrak{m}_{\dot{\nu}}$ are non-isomorphic simple $D(T)$-modules when $\dot{\nu} \mapsto \nu$ varies. Their direct sum is thus a simple $\mathrm{W} \ltimes D(T)$-module, as asserted.
\end{proof}

\subsection{A result of decompletion}\label{subsec:decompletion-W-regular}
We will use the formalism of Subsection \ref{subsec:inf-character}, especially the decomposition of the category $\cate{Adm} = \bigoplus_{[\nu]} \cate{Adm}_{[\nu]}$ according to infinitesimal characters; see Proposition \ref{prop:Adm-decomp}.

Define the full subcategory
\begin{equation*}
	\cate{Adm}_{\text{reg}} := \bigoplus_{[\nu]: \, \widetilde{\mathrm{W}}\text{-regular}} \cate{Adm}_{[\nu]}
\end{equation*}
of $\cate{Adm}$. It is still a Serre subcategory.

\begin{theorem}\label{prop:decompletion-reg}
	Let $M$ be an object of $\cate{Adm}_{\text{reg}}$, equipped with a good filtration. Then
	\[ \widehat{M} = 0 \iff \gr(M) = 0 \iff M = 0. \]
\end{theorem}
\begin{proof}
	The first equivalence is just Lemma \ref{prop:null-completion}, so it remains to show $M \neq 0 \implies \widehat{M} \neq 0$. We first address the case when $M$ is a nonzero quotient of some $\mathcal{G}_\nu$. By assumption $\nu$ is $\widetilde{\mathrm{W}}$-regular, hence $M \simeq \mathcal{G}_\nu$ by Theorem \ref{prop:simplicity}. Therefore $\widehat{M} \simeq \widehat{\mathcal{G}}_\nu$, which is nonzero by Proposition \ref{prop:Gnu-nonzero}.
	
	For general $M$, take a filtration $\{0\} = M_0 \subsetneq \cdots \subsetneq M_n = M$ as in Remark \ref{rem:HC-admissible}; every subquotient $M_i/M_{i-1}$ is a nonzero quotient of $\mathcal{G}_{\nu_i}$ for some $\nu_i \in \mathfrak{c}$, for $1 \leq i \leq n$.
	
	Note that $\nu_i$ is $\widetilde{\mathrm{W}}$-regular for all $i$. Hence the previous step implies $(M_i/M_{i-1})^{\wedge} \neq 0$ for all $i$. In view of the exactness of completion, we deduce that $\widehat{M} \neq 0$.
\end{proof}

\begin{theorem}\label{prop:minimal-extension}
	Let $M$ be an object of $\cate{Adm}_{\text{reg}}$, and let $M^\natural$ be the corresponding object of $(D(G), \mathfrak{n}_{\LeRi}^{\bpsi})\dcate{Mod}$ via \eqref{eqn:monodromic-equivalence}. Denote by $j$ the open embedding $Bw_0 B \hookrightarrow G$. Then $j^* M^\natural$ is a flat connection on $Bw_0 B$, and there is a canonical isomorphism of $D(G)$-modules
	\[ M^\natural \rightiso j_{!*} \left( j^* M^\natural \right). \]
\end{theorem}
\begin{proof}
	In view of Remark \ref{rem:minimal-extension}, it suffices to show that $M$ has neither quotients nor submodules that are non-zero of $\mathcal{S}$-torsion. But $\cate{Adm}_{\text{reg}}$ is closed under subquotients, so combining Corollary \ref{prop:torsion-cor} and Theorem \ref{prop:decompletion-reg} yields the desired result.
\end{proof}

In particular, Theorem \ref{prop:minimal-extension} applies to the Harish-Chandra modules $\mathcal{G}_\nu$ such that $\nu$ is $\widetilde{\mathrm{W}}$-regular.


\printbibliography[heading=bibintoc]

\vspace{1em}
\begin{flushleft} \small
	W.-W. Li: Beijing International Center for Mathematical Research / School of Mathematical Sciences, Peking University. No.\ 5 Yiheyuan Road, 100871 Beijing, People's Republic of China. \\
	E-mail address: \href{mailto:wwli@bicmr.pku.edu.cn}{\texttt{wwli@bicmr.pku.edu.cn}}
\end{flushleft}

\end{document}